\documentclass[reqno]{amsart}

\usepackage[left=26mm,top=30mm,right = 26mm, bottom = 30mm]{geometry}

\usepackage[english]{babel} 
\usepackage[utf8]{inputenc} 
\usepackage[T1]{fontenc}
\usepackage{tikz-cd}
\usepackage{enumitem}
\usepackage[backgroundcolor=blue!5!white]{todonotes}
\usepackage{mathtools} 
\usepackage{hyperref}
\usepackage[capitalize]{cleveref} 
\usepackage{mathrsfs}
\usepackage{quiver}
\usepackage{bussproofs}
\usepackage{mleftright}

\newtheorem{theorem}{Theorem}[section]
\newtheorem{lemma}[theorem]{Lemma}
\newtheorem{proposition}[theorem]{Proposition}
\newtheorem{corollary}[theorem]{Corollary}

\newtheorem{claim}[theorem]{Claim}
\newtheorem{fact}[theorem]{Fact}
\theoremstyle{definition}
\newtheorem{definition}[theorem]{Definition}
\newtheorem{notation}[theorem]{Notation}
\newtheorem{example}[theorem]{Example}
\newtheorem{question}[theorem]{Question}
\newtheorem{remark}[theorem]{Remark}

\numberwithin{equation}{section}

\newcommand{\cat}[1]{\mathsf{#1}}

\newcommand{\Set}{\cat{Set}}
\newcommand{\BA}{\cat{BA}}
\newcommand{\Pos}{\cat{Pos}}
\newcommand{\ctx}{\cat{Ctx}}
\newcommand{\D}{\cat{D}}
\newcommand{\C}{\mathsf{C}}

\newcommand{\QFF}{\cat{QFF}}
\newcommand{\QA}{\cat{QA}}
\newcommand{\DoctBA}{\cat{Doct}_{\BA}}
\newcommand{\DoctABA}{\cat{Doct}_{\EA\BA}}

\newcommand{\LT}{\mathsf{LT}}

\newcommand{\N}{\mathbb{N}}

\renewcommand{\L}{\mathcal{L}}
\newcommand{\T}{\mathcal{T}}

\newcommand{\Hom}{\mathrm{Hom}}

\newcommand{\tmn}{\mathbf{1}}

\renewcommand{\P}{\mathbf{P}}
\newcommand{\R}{\mathbf{R}}
\renewcommand{\H}{\mathbf{H}}

\newcommand{\fa}[2]{(\forall{#1})_{#2}}
\newcommand{\ex}[2]{(\exists{#1})_{#2}}
\newcommand{\EA}{{\forall\mkern-2mu\exists}}

\newcommand{\ple}[1]{\langle#1\rangle}

\renewcommand{\S}{\mathcal{S}}
\newcommand{\B}{\mathcal{B}}
\newcommand{\Q}{\mathcal{Q}}

\newcommand{\m}{\mathfrak{m}}
\newcommand{\n}{\mathfrak{n}}

\DeclareMathOperator{\id}{id}
\DeclareMathOperator{\pr}{pr}
\newcommand{\op}{^{\mathrm{op}}}

\title[Quantifier-free formulas and quantifier alternation depth in doctrines]{Quantifier-free formulas and quantifier alternation depth\\ in doctrines}

\hypersetup{
    pdftitle={},
    pdfauthor={},
    pdfmenubar=false,
    pdffitwindow=true,
    pdfstartview=FitH,
    colorlinks=true,
    linkcolor=teal,
    citecolor=magenta,
    urlcolor=cyan
}



\author{Marco Abbadini}
\address{University of Birmingham, School of Computer Science,
B15 2TT Birmingham (UK)}
\email{m.abbadini@bham.ac.uk}

\author{Francesca Guffanti}
\address{University of Luxembourg,
Department of Mathematics,
6, avenue de la Fonte,
L-4364 Esch-Sur-Alzette (Luxembourg)}
\email{francesca.guffanti@uni.lu}

\begin{document}

\begin{abstract}
    This paper aims to incorporate the notion of quantifier-free formulas modulo a first-order theory and the stratification of formulas by quantifier alternation depth modulo a first-order theory into the algebraic treatment of classical first-order logic.

    The set of quantifier-free formulas modulo a theory is axiomatized by what we call a \emph{quantifier-free fragment} of a Boolean doctrine with quantifiers. Rather than being an intrinsic notion, a quantifier-free fragment is an additional structure on a Boolean doctrine with quantifiers. 
    Under a smallness assumption, the structures occurring as quantifier-free fragments of some Boolean doctrine with quantifiers are precisely the Boolean doctrines (without quantifiers).
    In particular, every Boolean doctrine over a small category is a quantifier-free fragment of its quantifier completion.

    Furthermore, the sequences obtained by stratifying an algebra of formulas by quantifier alternation depth modulo a theory are axiomatized by what we call \emph{QA-stratified Boolean doctrines}.
    While quantifier-free fragments are defined in relation to an ``ambient'' Boolean doctrine with quantifiers, a QA-stratified Boolean doctrine requires no such ambient doctrine, and it consists of a sequence of Boolean doctrines (without quantifiers) with connecting axioms.
    QA-stratified Boolean doctrines are in one-to-one correspondence with pairs consisting of a Boolean doctrine with quantifiers and a quantifier-free fragment of it.
\end{abstract}

\maketitle

\setcounter{tocdepth}{3}
\makeatletter
\def\l@subsection{\@tocline{2}{0pt}{2.5pc}{5pc}{}}
\def\l@subsubsection{\@tocline{2}{0pt}{5pc}{7.5pc}{}}
\makeatother

\tableofcontents

\section{Introduction}

A common measure of the complexity of a first-order formula is its quantifier alternation depth, which counts the number of blocks of alternating existential and universal quantifiers. 
This gives a stratification $\mathcal{F}_0 \subseteq \mathcal{F}_1 \subseteq \mathcal{F}_2 \subseteq \dots$ of the set $\mathcal{F}$ of all first-order formulas, where $\mathcal{F}_n$ consists of the formulas whose quantifier alternation depth is at most $n$. 
For example, an atomic formula $R(x,y)$ is in $\mathcal{F}_0$, the formula $\forall x\, R(x,y)$ belongs to $\mathcal{F}_1$, and the formula $\exists y\, \forall x\, R(x,y)$ belongs to $\mathcal{F}_2$.

The main objective of this paper is to incorporate the stratification induced by the syntactic notion of quantifier alternation depth into the algebraic treatment of first-order logic.
Among various algebraizations of first-order logic, we choose to work with
Lawvere's hyperdoctrines \cite{Lawvere69,Lawvere70}, of categorical flavour.
In contrast to other algebraic approaches to first-order logic (such as polyadic \cite{Halmos1956} or cylindric algebras \cite{HenkinMonkEtAl1971}), formulas in hyperdoctrines are indexed by their set of free variables.
This allows to faithfully reflect the fact that quantifiers change the set of free variables.
To cover a range of different logical fragments and rules, many variations of Lawvere's original definition of hyperdoctrines were introduced.
Among these, we place ourselves in the setting of \emph{first-order Boolean doctrines}, which capture the Boolean setting with quantifiers.\footnote{Points of departure with Lawvere's original definition \cite{Lawvere69,Lawvere70} are, among others, the fact that we require the fibers to be Boolean algebras rather than Cartesian closed categories, that we do not require the presence of the equality predicate (which we treat separately in \cref{s:equality}), and that we do not require the base category to be Cartesian closed.}

Our source of inspiration has been M.~Gehrke's talk at the conference ``Category theory 20$\to$21'';
this work is our attempt to take up on her invitation \cite[minute 55]{GehrkeTalk}:
\begin{quote}
    ``I wish you would try to make some nice mathematics [...] or some nice category theory out of this decoupage of the formulas [= the stratification by quantifier depth] rather than just this [= the Lawverian stratification by contexts]. I mean, [the Lawverian way] is important, but [the stratification by quantifier depth] is very useful, technically.''
\end{quote}

\noindent (It is worth mentioning that in our work we consider the quantifier alternation depth rather than the notion of depth of nesting of quantifiers. 
To illustrate the difference, consider the formula $\forall x\, \forall y\, R(x, y)$: it has quantifier alternation depth $1$ but depth of nesting of quantifiers $2$. 
We take this approach because, in the Lawverian setting, there is no notion of ``quantification on exactly one variable''.)

Because of the practice in mathematics of working modulo a theory, rather than algebraizing the ``pure'' quantifier alternation depth, we provide a doctrinal approach to the quantifier alternation depth \emph{modulo a theory}.
In doing so, we face the issue that, in doctrines, any information about this depth is lost. Indeed, if two theories over possibly different languages have isomorphic doctrines of formulas, the isomorphism may fail to preserve quantifier-free
formulas. To solve this problem, instead of seeking an intrinsic notion of quantifier alternation depth, we shall add further structure to a first-order Boolean doctrine.

To this end, we give two definitions, which carry the same information: we define\ldots
\begin{enumerate}
    \item \ldots a \emph{quantifier-free fragment} of a first-order Boolean doctrine: this axiomatizes the subset $\mathcal{F}_0$ of quantifier-free formulas modulo $\T$ of a given set $\mathcal{F}$ of first-order formulas modulo a theory $\T$; 

    \item \ldots a \emph{QA-stratified Boolean doctrine}: it axiomatizes the stratification $\mathcal{F}_0\subseteq\mathcal{F}_1\subseteq\mathcal{F}_2\subseteq \dots$ of the set of formulas by quantifier-alternation depth modulo a theory. The ``ambient'' set of all formulas is not mentioned in the axiomatization, but can be obtained as the directed union of all the layers.
\end{enumerate}
The two notions define the objects of two categories, which we show to be equivalent (\cref{s:q-depth}).

Then, we turn to the question: if the structure on the set $\mathcal{F}$ of all first-order formulas modulo $\T$ is a first-order Boolean doctrine, what is the structure on the set $\mathcal{F}_0$ of quantifier-free formulas modulo $\T$?
We show that (under a smallness assumption) the structures occurring as quantifier-free fragments of a first-order Boolean doctrine are precisely the \emph{Boolean doctrines}, which are a version of first-order Boolean doctrines that does not require the existence of quantifiers (\cref{sec:char-qff}).

In future work, we aim to use the step-by-step approach based on the quantifier alternation depth of the previous sections to describe how to freely add quantifiers to a Boolean doctrine.
This should be similar to the description of free modal algebras and free Heyting algebras with the step-by-step approach based on the depth of nesting of modalities and implications \cite{Ghilardi1992,Ghilardi1995,BezhanishviliGehrke2011,Ghilardi2010,CoumansVanGool2013}.
To prepare the ground for this exploration (of which some first results are in \cite{AbbadiniGuffanti2}), in this paper we observe that any Boolean doctrine over a small category freely generates a first-order Boolean doctrine (its \emph{quantifier completion}), of which, moreover, it is a quantifier-free fragment (\cref{sec:quantifier compl}).

Finally, the setting with equality is discussed in \cref{s:equality}. A pleasant realization is that the situation does not change much: in particular, a quantifier completion of a Boolean doctrine with equality has equality.

To conclude, we sum up the main results of the paper:
\begin{enumerate}
    \item We provide a quasi-equational axiomatization of the algebras of first-order theories where the formulas are separated according to quantifier alternation depth and context, instead of only their context (\cref{d:only-layers,t:QA-makes-sense}).

    \item
    We show that every Boolean doctrine embeds into some first-order Boolean doctrine (\cref{t:bd-embeds-fo}), and, in particular, into its free completion under existential and universal quantifiers; this characterizes Boolean doctrines as the algebras of quantifier-free formulas (\cref{t:character-0th-layer,p:qff-of-completion}).

    \item
    We show that equality, when it exists, is preserved by quantifier completion (\cref{c:equality-only-p0}).
\end{enumerate}

Hopefully, this paper will be a first step for a doctrinal approach to quantifier-free formulas and the quantifier alternation depth modulo a theory.

\section{Preliminaries on doctrines}

Hyperdoctrines were introduced by F.~W.~Lawvere \cite{Lawvere69,Lawvere70} to interpret both syntax and semantics of first-order theories in the same categorical setting. 
In this paper, we consider Boolean doctrines and first-order Boolean doctrines, which are variations of Lawvere's hyperdoctrines; points of departure are, among others, the fact that we impose all the axioms of Boolean algebras and that we do not require the existence of the equality predicate. (We treat the equality predicate separately in \cref{s:equality} via the notion of \emph{elementarity}.)
Boolean doctrines contain enough structure to interpret all logical connectives, while first-order Boolean doctrines require further structure allowing to interpret also the quantifiers.
Lawvere's fundamental intuition was that quantifiers in logic are interpreted as certain adjoints.

Lawvere's doctrinal setting is amenable to a number of generalizations different from ours; for the interested reader we mention primary doctrines (where one can interpret $\top$ and $\land$) \cite{MaiettiRosolini13b,Pasquali15,EmPaRo20}, existential doctrines ($\top$, $\land$, $\exists$) \cite{MaiettiRosolini13b}, universal doctrines ($\top$, $\land$, $\forall$) \cite{Pasquali15}, elementary doctrines ($\top$, $\land$, $=$) \cite{MaiettiRosolini13,EmPaRo20} and first-order doctrines (all logical connectives with the axioms of Heyting algebras and quantifiers) \cite{EmPaRo20}.

\begin{notation}
    We let $\BA$ denote the category of Boolean algebras and Boolean homomorphisms, and $\Pos$ the category of partially ordered sets and order-preserving functions.
\end{notation}

\begin{definition}[Boolean doctrine]
    A \emph{Boolean doctrine} over a category $\C$ with finite products is a functor $\P \colon \C\op \to  \BA$.
    The category $\C$ is called the \emph{base category of $\P$}.
    For each $X\in\C$, $\P(X)$ is called a \emph{fiber}. For each morphism $f\colon X'\to X$, the function $\P(f)\colon\P(X)\to\P(X')$ is called the \emph{reindexing along $f$}.
\end{definition}

\begin{definition}[Boolean doctrine morphism]
     A \emph{Boolean doctrine morphism} from $\P\colon\C\op\to \BA$ to $\mathbf{R}\colon \cat{D}\op\to \BA$ is a pair $(M,\m)$ where $M\colon \C\to\cat{D}$ is a functor that preserves finite products and $\m \colon \P\to \mathbf{R}\circ M\op $ is a natural transformation.
    \[
        \begin{tikzcd}
            \C\op && {\cat{D}\op} \\
            \\
            & \BA
            \arrow["M\op", from=1-1, to=1-3]
            \arrow[""{name=0, anchor=center, inner sep=0}, "\P"', from=1-1, to=3-2]
            \arrow[""{name=1, anchor=center, inner sep=0}, "\R", from=1-3, to=3-2]
            \arrow["\m", curve={height=-6pt}, shorten <=8pt, shorten >=8pt, from=0, to=1]
        \end{tikzcd}
    \]

    The composite Given Boolean doctrine morphisms $(M, \m) \colon \P\to\mathbf{R}$ and $(N,\mathfrak n) \colon \R \to \mathbf{S}$,
    \[
        \begin{tikzcd}
    	\C\op && {\cat{D}\op} && {\cat{E}\op} \\
    	\\
    	&& \BA
    	\arrow["M\op", from=1-1, to=1-3]
    	\arrow[""{name=0, anchor=center, inner sep=0}, "\P"', from=1-1, to=3-3]
    	\arrow[""{name=1, anchor=center, inner sep=0}, "\R", from=1-3, to=3-3]
    	\arrow["N\op", from=1-3, to=1-5]
    	\arrow[""{name=2, anchor=center, inner sep=0}, "{\mathbf{S}}", from=1-5, to=3-3]
    	\arrow["\m", curve={height=-6pt}, shorten <=8pt, shorten >=8pt, from=0, to=1]
    	\arrow["\n", curve={height=-6pt}, shorten <=8pt, shorten >=8pt, from=1, to=2]
        \end{tikzcd}
    \]
    their composite $(N, \n) \circ (M, \m) \colon \P \to \mathbf{S}$ is the pair $(N \circ M, \mathfrak{n}\circ\m) \colon \P \to \mathbf{S}$, where $N \circ M$ is the composite of the functors between the base categories, and the component at $X\in \C$ of the natural transformation $\mathfrak{n}\circ\m$ is defined as $(\mathfrak{n}\circ\m)_X = \mathfrak{n}_{M(X)}\circ\m_X$, i.e.\ the composite of the following functions:
    \[
    \P(X)\xrightarrow{\m_X}\mathbf{R}(M(X))\xrightarrow{\mathfrak{n}_{M(X)}}\mathbf{S}(NM(X)).
    \]
\end{definition}

Although 2-categorical aspects of doctrines would be very natural, we omit them for simplicity.

\begin{definition}[First-order Boolean doctrine]\label{d:bool_ex_doc}
    A \emph{first-order Boolean doctrine} over a category $\C$ with finite products is a functor $\P \colon \C\op \to  \BA$ with the following properties.
    \begin{enumerate}
        \item {(Universal)} \label{i:h3}
        For all $X, Y \in \C$, letting $\pr_1 \colon X \times Y \to X$ denote the projection onto the first coordinate, the function
        \[
        \P(\pr_1) \colon \P(X) \to \P(X \times Y)
        \]
        has a right adjoint $\fa{Y}{X}$ (as an order-preserving map between posets).
        This means that for every $\beta \in \P(X \times Y)$ there is a (necessarily unique) element $\fa{Y}{X} \beta \in \P(X)$ such that, for every $\alpha \in \P(X)$, 
        \[
            \alpha \leq \fa{Y}{X} \beta \ \text{ in } \P(X) \quad\iff \quad \P(\pr_1)(\alpha) \leq \beta \ \text{ in }\P(X \times Y).
        \] 
        
        \item
        (Beck-Chevalley) For every morphism $f\colon X'\to X$ in $\C$, the following diagram in $\Pos$ commutes. 
        \[
            \begin{tikzcd}
            {X} & {\P(X\times Y)} & {\P(X)} \\
            X' & {\P(X'\times Y)} & {\P(X')}
            \arrow["{\P(f\times\id_{Y})}", from=1-2, to=2-2, swap]
            \arrow["{\P(f)}", from=1-3, to=2-3]
            \arrow["{\fa{Y}{X'}}"', from=2-2, to=2-3]
            \arrow["{\fa{Y}{X}}", from=1-2, to=1-3]
            \arrow["f", from=2-1, to=1-1]
            \end{tikzcd}
        \]
    \end{enumerate}
\end{definition}

\begin{definition}[First-order Boolean doctrine morphism] \label{d:ABA-morphism}
     A \emph{first-order Boolean doctrine morphism} from $\P\colon\C\op\to \BA$ to $\mathbf{R} \colon \cat{D}\op\to \BA$ is a Boolean doctrine morphism $(M,\m)\colon \P\to\R$ such that for every $X,Y\in\C$ the following diagram commutes.
     \begin{equation*}
         \begin{tikzcd}
             \P(X\times Y)\arrow[d,"\fa{Y}{X}"']\arrow[r,"\m_{X\times Y}"] & \R(M(X)\times M(Y))\arrow[d,"\fa{M(Y)}{M(X)}"]\\
             \P(X)\arrow[r,"\m_{X}"'] & \R(M(X))
         \end{tikzcd}
     \end{equation*}
\end{definition}

\begin{remark}[Universality is equivalent to existentiality]
    In \cref{d:bool_ex_doc}\eqref{i:h3}, we required the existence of the right adjoint $\fa{Y}{X}$ of $\P(\pr_1)$.
    Alternatively, one may ask for the existence of a left adjoint $\ex{Y}{X}$ (again with the Beck-Chevalley condition), a property called \emph{existentiality}.
    In this Boolean case, existentiality and universality are equivalent, since the two quantifiers are interdefinable: $\forall = \lnot \exists \lnot$ and $\exists = \lnot \forall \lnot$.
\end{remark}
    
Next, we describe the leading example: the first-order Boolean doctrine that describes a first-order theory.

\begin{example}[Syntactic doctrine]\label{fbf}
    Fix a first-order language $\L = (\mathbb{F},\mathbb{P})$ (without equality) and a theory $\T$ in the language $\L$.
    We define a first-order Boolean doctrine 
    \[
    \LT^{\L,\T} \colon \ctx \op\to \BA,
    \]
    called the \emph{syntactic doctrine of ($\L$ and) $\T$},
    as follows. ($\LT$ stands for ``Lindenbaum-Tarski algebra''.)
    An object of the base category $\ctx$ is a finite list of distinct variables (also called \emph{context}), and a morphism from $\vec x=(x_1,\dots, x_n)$ to $\vec y=(y_1,\dots, y_m)$ is an $m$-tuple
    \begin{equation*}
        (t_1(\vec x),\dots,t_m(\vec x))\colon (x_1,\dots, x_n) \to (y_1,\dots, y_m)
    \end{equation*}
    of terms in the context $\vec x$. The composition of morphisms in $\ctx$ is given by simultaneous substitutions. We remark that the category $\ctx$ depends only on the set $\mathbb{F}$ of function symbols of the language $\L$.
    On objects, the functor $\LT^{\L,\T} \colon \ctx\op\to \BA$ maps a context $\vec x$ to the poset reflection of the preordered set of formulas with at most those free variables, ordered by provable consequence $\vdash_\T$ in $\T$, according to which $\alpha$ is below $\beta$ if and only if the sequent $\alpha \Rightarrow_{\vec x} \beta$ is provable from $\T$; here, the subscript $\vec x$ in the sequent symbol $\Rightarrow$ means that the sequent is considered in the context $\vec x$.
    Although this is folklore, we recall in \cref{app:calculus} the rules of the sequent calculus with contexts for classical first-order logic. (A tangible consequence of the slight difference with the more usual logic \emph{without} contexts is that the sequent $\Rightarrow_{()} \exists x \top$ is not provable, in accordance with admitting the empty set as a possible model.)
    On morphisms, the functor $\LT^{\L,\T}$ maps a morphism $\vec{t}(\vec{x}) \colon \vec{x}\to\vec{y}$ to the substitution $[\vec{t}(\vec{x})/\vec{y}] \colon \LT^{\L,\T}(\vec y) \to \LT^{\L,\T}(\vec x)$. 
    
    The functor $\LT^{\L,\T}$ is a first-order Boolean doctrine.
    In particular, given finite lists of variables $\vec x$ and $\vec y$ with no common variables, and letting $\pr_1$ denote the projection morphism $\vec x\colon(\vec x,\vec y)\to\vec x$, the right adjoint to $\LT^{\L,\T}(\pr_1)\colon\LT^{\L,\T}(\vec x)\to\LT^{\L,\T}(\vec x,\vec y)$ is 
    \[
    \forall y_1\,\dots\,\forall y_m\colon\LT^{\L,\T}(\vec x,\vec y)\to\LT^{\L,\T}(\vec x).
    \]
    
    If no confusion arises, instead of $\LT^{\L,\T}$ we simply write $\LT^{\T}$, omitting the superscript ``$\L$''.
\end{example}

With this example in mind, we suggest thinking of the objects of the base category of a first-order Boolean doctrine as lists of variables, the morphisms as terms, the fibers as sets of formulas, the reindexings as substitutions, the Boolean operations as logical connectives, and the adjunctions between fibers as quantifiers.

\begin{remark}
    The setting of first-order Boolean doctrines encompasses in a similar way also \emph{many-sorted} first-order theories.
\end{remark}

\begin{example}[Subset doctrine]
    The \emph{subset doctrine}\footnote{The subset doctrine is useful for encoding models in the doctrinal setting.
    Indeed, a model of a first-order theory $\T$ (in the classical sense) corresponds precisely to a first-order Boolean doctrine morphism $(M, \m)$ from the syntactic doctrine $\LT^\T$ to the subset doctrine $\mathscr{P}$.
    The assignment of the functor $M$ on objects encodes the underlying set of the model, the assignment of $M$ on morphisms encodes the interpretation of the function symbols, and the natural transformation $\m$ encodes the interpretation of the relation symbols.} is the contravariant power set functor $\mathscr{P} \colon \Set\op \to \BA$, which maps a set $X$ to its power set Boolean algebra $\mathscr{P}(X)$, and $f\colon X'\to X$ to the preimage function
    \[
    \mathscr{P}(f) \coloneqq f^{-1}[-] \colon \mathscr{P}(X) \to \mathscr{P}(X').
    \]
    This is a first-order Boolean doctrine. Given two sets $X$ and $Y$, the right adjoint $\fa{Y}{X}$ to the order-preserving function $\pr_1^{-1}[-]\colon\mathscr{P}(X)\to\mathscr{P}(X\times Y)$ is the function
    \begin{align*}
        \fa{Y}{X} \colon \mathscr{P}(X\times Y) & \longrightarrow \mathscr{P}(X)\\
        S & \longmapsto \{x \in X \mid \text{for all }y \in Y, \, (x, y) \in S\}.
    \end{align*}
    In this case, the left adjoint $\ex{Y}{X}$ is easier to describe, and is the direct image function of $\pr_1 \colon X \times Y \to X$.
\end{example}

\subsection{Doctrines as many-sorted algebras}

In the following, we present Boolean doctrines and first-order Boolean doctrines over a fixed small category as many-sorted algebras forming a class defined by equations.
While most of the general references on universal algebra treat the case of a single sort, most if not all of the elementary theory generalizes to multiple sorts; we refer to \cite{Higgins1963}, \cite{BirkhoffLipson1970}, \cite[Sec.~4.1.1]{Wechler1992}, \cite[Ch.~14]{AdamekRosickyEtAl2011} and \cite{Tarlecki2011}.

\begin{definition}
    Let $\C$ be a category with finite products.
    \begin{enumerate}
        \item 
        We let $\DoctBA^\C$ denote the category of Boolean doctrines over $\C$ and natural transformations.
        
        \item 
        We let $\DoctABA^\C$ denote the category whose objects are first-order Boolean doctrines over $\C$ and whose morphisms from $\P \colon \C\op \to \BA$ to $\R \colon \C\op \to \BA$ are natural transformations $\m \colon \P \to \R$ such that for all $X, Y \in \C$ the following diagram commutes.
        \begin{equation} \label{eq:commutative-algebras-diagram}
            \begin{tikzcd}
             \P(X \times Y) \arrow[d,"\fa{Y}{X}"']\arrow[r,"\m_{X\times Y}"] & \R(X \times Y) \arrow[d,"\fa{Y}{X}"]\\
             \P(X)\arrow[r,"\m_{X}"'] & \R(X)
            \end{tikzcd}
        \end{equation}
    \end{enumerate}
\end{definition}

\begin{remark}[Boolean doctrines and first-order Boolean doctrines as many-sorted algebras] \label{r:many-sorted}
    Let $\C$ be a small\footnote{Although smallness is irrelevant for many algebraic arguments, it guarantees the existence of free algebras (used in \cref{t:left-adjoint}).} category with finite products.
    We present $\DoctABA^\C$ as a variety of many-sorted algebras.
    
    First, we describe a many-sorted algebraic language $\L_\C$.
    The sorts are the objects of $\C$.
    We equip each sort with the signature of a Boolean algebra.
    Moreover, for each morphism $f \colon X \to Y$ in $\C$, we consider a unary function symbol $f$ from sort $Y$ to sort $X$.
    Finally, for each binary product diagram $X \xleftarrow{\pr_1} Z \xrightarrow{\pr_2} Y$, we consider a unary function symbol $\fa{Y}{X}^{\pr_1,\pr_2}$ from sort $Z$ to sort $X$.
    
    Next, we let $\mathcal{V}$ be the class of many-sorted algebras $\P$ in the language $\L_\C$ satisfying the following equational axioms, where for each object $X$ we write $\P_X$ for the value of $\P$ at the sort $X$, and for each morphism $f \colon X \to Y$ we write $\P_f$ for the interpretation of the function symbol $f$ in $\P$.
    \begin{enumerate}
        \item \label{i:sort-is-bool}
        Each sort satisfies the axioms of a Boolean algebra.
        
        \item \label{i:maps-are-homo}
        For each morphism $f \colon X \to Y$, the function symbol $\P_f$ satisfies the axioms of a Boolean homomorphism, i.e.\ for each Boolean function symbol $g(x_1, \dots, x_n)$, we have the axiom
        \[
            \text{For all }\alpha_1, \dots, \alpha_n \in \P_Y, \ \P_f(g_{\P_Y}(\alpha_1, \dots, \alpha_n)) = g_{\P_X}(\P_f(\alpha_1), \dots, \P_f(\alpha_n)).
        \]
        
        \item \label{i:id-preserved}
        For each object $X$, we have the axiom 
        \[
        \text{For all } \alpha \in \P_X, \  \P_{\id_X}(\alpha) = \alpha.
        \]
        
        \item \label{i:comp-preserved}
        Given two morphisms $X \xrightarrow{f} Y \xrightarrow{g} Z$, we have the axiom
        \[
            \text{For all } \alpha \in \P_Z,\  \P_f (\P_g(\alpha)) = \P_{g \circ f} (\alpha).
        \]
        
        \item \label{i:universal-adjoint}
        For each binary product diagram $X \xleftarrow{\pr_1} Z \xrightarrow{\pr_2} Y$, we have the axiom
        \[
            \text{For all } \alpha \in \P_X \text{ and }\beta \in \P_Z,\ \alpha \leq \fa{Y}{X}^{\pr_1,\pr_2} (\beta) \iff \P_{\pr_1}(\alpha) \leq \beta.
        \]
        
        \item \label{i:universal-BC}
        For all objects $X, X', Y, Z, Z'$, for every morphism $f \colon X' \to X$, and for all binary product diagrams $X \xleftarrow{\pr_1}Z \xrightarrow{\pr_2} Y$ and $X' \xleftarrow{\pr_1'}Z' \xrightarrow{\pr_2'} Y$, we have the axiom
        \[
            \text{For all }\alpha \in \P_{Z},\ \fa{Y}{X'}^{\pr_1',\pr_2'} (\P_{f \times \id_Y}(\alpha)) = \P_f(\fa{Y}{X}^{\pr_1,\pr_2}(\alpha)).
        \]
    \end{enumerate}

    We observe that the axiom in \eqref{i:universal-adjoint} can be expressed by equations.
    Indeed, for any pair $(L, R)$, where $L \colon P \to Q$ and $R \colon Q \to P$ are order-preserving functions between posets, the adjunction condition
    \[
    L(a) \leq b \iff a \leq R(b)
    \]
    is equivalent to the unit-counit formulation
    \[
    a\leq RL(a)\quad\text{ and }\quad LR(b)\leq b.
    \]
    Therefore, the condition in \eqref{i:universal-adjoint} can be expressed equivalently via the following equations.
    \begin{enumerate}
        \item (Order-preservation)
        For all $\beta,\beta' \in \P_Z$, $\fa{Y}{X}^{\pr_1,\pr_2}(\beta \land \beta') \leq \fa{Y}{X}^{\pr_1,\pr_2}(\beta)$.
        \item (Unit) For all $\alpha \in \P_X$, 
        $\alpha \leq \fa{Y}{X}^{\pr_1,\pr_2}(\P_{\pr_1}(\alpha))$.

        \item (Counit) For all $\beta \in \P_Z$, 
        $\P_{\pr_1}(\fa{Y}{X}^{\pr_1,\pr_2}(\beta))\leq\beta$.
    \end{enumerate}
    (Of course, every inequality can be expressed equationally since $a\leq b$ is equivalent to $a\land b =a$.)
    
    A many-sorted algebra in this signature satisfying the axioms above is the same thing as a first-order Boolean doctrine over $\C$.
    Indeed,
    \begin{itemize}
        \item 
        \eqref{i:sort-is-bool} guarantees that we have an assignment on objects from $\C$ to $\BA$.
        \item
        \eqref{i:maps-are-homo} guarantees that we have an assignment on morphisms from $\C$ to $\BA$.
        \item
        \eqref{i:id-preserved} guarantees that the identity is preserved, and \eqref{i:comp-preserved} guarantees that the composition is preserved, so that we have a functor $\C\op \to \BA$.
        \item
        \eqref{i:universal-adjoint} guarantees that the universal quantifier is a right adjoint.
        \item
        \eqref{i:universal-BC} guarantees that the Beck-Chevalley condition is satisfied.
    \end{itemize}
    
    A homomorphism $\m$ of many-sorted algebras in $\mathcal{V}$ is the same thing as a morphism in $\DoctABA^\C$.
    Indeed,
    \begin{itemize}
        \item The preservation of the Boolean function symbols guarantees that $\m_X$ is a Boolean homomorphism for each $X \in \C$.
        \item The preservation of the unary function symbols associated to the morphisms of $\C$ guarantees the naturality of $\m$.
        \item The preservation of the unary function symbols $\fa{Y}{X}^{\pr_1,\pr_2}$ guarantees the commutativity of \eqref{eq:commutative-algebras-diagram}.
    \end{itemize}

    Similarly, $\DoctBA^\C$ is a variety of many-sorted algebras in the sublanguage of $\L_\C$ consisting of all its function symbols except quantifiers.
\end{remark}

\subsection{The universal property of the syntactic doctrine} In the following, we state the universal property of the syntactic doctrine $\LT^{\T}$; essentially, this amounts to the fact that a morphism in $\DoctABA^\ctx$ with domain $\LT^{\T}$ is uniquely determined by its assignment on the equivalence class of atomic formulas.

\begin{remark}[Interpretation of formulas]\label{r:int-fmlas}
    Let $\L=(\mathbb{F},\bigcup_{n\in\N}\mathbb{P}_n)$ be a first-order language, where $\mathbb{P}_n$ denotes the set of $n$-ary predicate symbols. Let $\ctx$ be the category of contexts for the functional language $\mathbb{F}$ (defined as in \cref{fbf}). We consider an enumeration $x_1, x_2, \dots$ of the variables. Let $\P\colon\ctx\op\to\BA$ be a first-order Boolean doctrine and let $(\mathbb{I}_n\colon\mathbb{P}_n\to\P(x_1,\dots,x_n))_{n\in\N}$ be a family of functions.
    We define inductively on the complexity of the formulas a function $\mathcal{I}$ from the set 
    \[
    \mathrm{Form}(\L) \coloneqq \{ \alpha : \vec x\mid \vec x \in \ctx, \alpha \text{ first-order formula s.t } \mathrm{FV}(\alpha)\subseteq\{x_1,\dots,x_n\}\}
    \]
    of $\L$-formulas in context---where $\mathrm{FV}(\alpha)$ is the set of free variables of the formula $\alpha$---to $\bigcup_{\vec x \in\ctx}\P(\vec x)$ as follows (we let $n$ and $m$ be the lengths of the contexts $\vec x$ and $\vec y$ respectively):
    \begin{enumerate}
        \item\label{i:int-subs} 
        for all $S\in\mathbb{P}_m$ and for all morphisms $\vec{t}\colon\vec x \to\vec y$ in $\ctx$, we define
        \[\mathcal{I}(S(\vec{t}(\vec x)):\vec x)\coloneqq \P(\vec t)(\mathbb{I}_m(S))\in\P(\vec x);\]
        \item\label{i:int-bool} 
        for all $*\in\{\land,\lor\}$, for all $\alpha:\vec x$ and $\beta:\vec x$ we define 
        \[\mathcal{I}(\alpha*\beta:\vec x)\coloneqq \mathcal{I}(\alpha:\vec x)*\mathcal{I}(\beta:\vec x)\in\P(\vec x),\]
        for every $\alpha:\vec x$ we define 
        \[\mathcal{I}(\lnot\alpha:\vec x)\coloneqq \lnot\mathcal{I}(\alpha:\vec x)\in\P(\vec x),\]
        and we further define 
        \[\mathcal{I}(\top:\vec x)\coloneqq \top_{\P(\vec x)}\in\P(\vec x) \quad \text{ and } \quad \mathcal{I}(\bot:\vec x)\coloneqq \bot_{\P(\vec x)}\in\P(\vec x);\]
        \item\label{i:int-fa} for every $\alpha:(\vec  x, z)$ we define 
        \begin{align*}
        \mathcal{I}(\forall z\,\alpha:\vec x)\coloneqq \fa{(z)}{\vec x}\,\mathcal{I}(\alpha:(\vec x, z))\in\P(\vec x) \quad \text{ and } \quad
        \mathcal{I}(\exists z\,\alpha:\vec x)\coloneqq \ex{(z)}{\vec x}\,\mathcal{I}(\alpha:(\vec x, z))\in\P(\vec x).
        \end{align*}
    \end{enumerate}
    
    Roughly speaking, we ``extended'' the interpretation in $\P$ of the predicate symbols of $\L$ to the interpretation in $\P$ of every formula.

    Additionally, it can be proved by induction on the complexity of the formulas that for every $\alpha:\vec y$ and for every morphism $\vec{t}\colon\vec x \to\vec y$ in $\ctx$ we have
    \begin{equation}\label{eq:I-nat}
        \mathcal{I}(\alpha[\vec t(\vec x)/\vec y]:\vec x)=\P(\vec t)(\mathcal{I}(\alpha:\vec y)).
    \end{equation}
\end{remark}

\begin{theorem}[Universal property of $\LT^{\T}$]\label{p:univ-prop-ltlt}
    Let $\T$ be a first-order theory in a language $(\mathbb{F},\bigcup_{n\in\N}\mathbb{P}_n)$, with $\mathbb{P}_n$ the set of $n$-ary predicate symbols, and $\P\colon\ctx\op\to\BA$ a first-order Boolean doctrine, where $\ctx$ is the category of contexts for the functional language $\mathbb{F}$.
    The set of first-order Boolean doctrine morphisms of the form $(\id_\ctx,\mathfrak{i})\colon\LT^{\T}\to\P$ is in bijection with the set of families of functions $(\mathbb{I}_n\colon\mathbb{P}_n\to\P(x_1,\dots,x_n))_{n\in\N}$ such that $\mathcal{I}(\alpha:())=\top_{\P()}$ for all $\alpha\in\T$, where $\mathcal{I}$ is the extension of $(\mathbb{I}_n)_{n}$ as in \cref{r:int-fmlas}. 
    The bijection assigns to $(\id_\ctx,\mathfrak{i})\colon\LT^{\T}\to\P$ the family $(\mathbb{I}_n\colon\mathbb{P}_n\to\P(\vec x),\, R \mapsto \mathfrak{i}_{\vec x}([R(\vec x):\vec x]_\T))_{n}$, and to $(\mathbb{I}_n)_n$ the morphism $(\id_\ctx,\mathfrak{i})\colon\LT^{\T}\to\P$ with
    \begin{align*}
        \mathfrak{i}_{\vec x}\colon\LT^{\T}(\vec x)&\longrightarrow\P(\vec x)\\
    [\alpha:\vec x]_\T&\longmapsto\mathcal{I}(\alpha:\vec x).
    \end{align*}  
\end{theorem}

\begin{proof}
    Let $(\id_\ctx,\mathfrak{i})\colon\LT^{\T}\to\P$ be a first-order Boolean doctrine morphism; define for every $n\in\N$ the function $\mathbb{I}_n\colon\mathbb{P}_n\to\P(\vec x),\, R \mapsto \mathfrak{i}_{\vec x}([R(\vec x):\vec x]_\T)$ and let $\mathcal{I}$ be its extension. It can be easily shown by induction on the complexity of the formulas that $\mathcal{I}(\alpha:\vec x)=\mathfrak{i}_{\vec x}([\alpha:\vec x]_\T)$ for every $\alpha:\vec x\in\mathrm{Form(\L)}$, where $\L$ is the language $(\mathbb{F},\bigcup_{n\in\N}\mathbb{P}_n)$. In particular, for every $\alpha\in\T$ we have $\mathcal{I}(\alpha:())=\mathfrak{i}_{()}([\alpha:()]_\T)=\mathfrak{i}_{()}([\top:()]_\T)=\top_{\P()}$.
    
    We now prove the well-definedness of the converse assignment. Let $(\mathbb{I}_n)_n$ be a family and $\mathcal{I}$ its extension. For each $\vec x\in\ctx$, we set $\mathfrak{i}_{\vec x}\colon\LT^{\T}(\vec x)\to\P(\vec x),
    [\alpha:\vec x]_\T\mapsto\mathcal{I}(\alpha:\vec x)$. This function is well-defined. Indeed,
    suppose $[\alpha:\vec x]_\T\leq[\beta:\vec x]_\T$. There is a proof tree of the sequent $\alpha\Rightarrow_{\vec x}\beta$ whose non-axiom leaves are sequents of the form $\Rightarrow_{()}\varphi$ with $\varphi\in\T$, which are valid in $(\P,\mathcal{I})$ (in the sense of \cref{d:valid-sequent}) since $\mathcal{I}(\varphi:())=\top_{\P_{()}}$. By \cref{l:rules-calc}, 
    every sequent in the tree is valid in $(\P,\mathcal{I})$, and in particular the root, which means that $\mathcal{I}(\alpha:\vec x)\leq\mathcal{I}(\beta:\vec x)$ in $\P(\vec x)$.
    Moreover, every component of $\mathfrak{i}$ is a Boolean homomorphism (by \cref{r:int-fmlas}\eqref{i:int-bool}). By eq.~\eqref{eq:I-nat}, it follows that $\mathfrak{i}\colon\LT^{\T}\to\P$ is a natural transformation.
    To conclude, by \cref{r:int-fmlas}\eqref{i:int-fa}, $\mathfrak{i}$ preserves quantifiers, and thus $(\id_\ctx,\mathfrak{i})$ is a first-order Boolean doctrine morphism.

    Next, we show that the two assignments are mutual inverses. 
    Let $(\id_\ctx,\mathfrak{i})\colon\LT^{\T}\to\P$ be a first-order Boolean doctrine morphism. We define for every $n\in\N$ the function $\mathbb{I}_n\colon\mathbb{P}_n\to\P(\vec x),\, R \mapsto \mathfrak{i}_{\vec x}([R(\vec x):\vec x]_\T)$. Next, let $\mathcal{I}$ be the extension of $(\mathbb{I}_n)_n$, and then let $\mathfrak{j}_{\vec x}\colon\LT^{\T}(\vec x)\to\P(\vec x),\, [\alpha:\vec x]_\T\mapsto\mathcal{I}(\alpha:\vec x)$ be the corresponding natural transformation.
    As mentioned in the first paragraph of the proof, by induction on the complexity of the formulas we get $\mathcal{I}(\alpha:\vec x)=\mathfrak{i}_{\vec x}([\alpha:\vec x]_\T)$, so that $\mathfrak{j}=\mathfrak{i}$, as desired.
    Conversely, let $(\mathbb{I}_n)_n$ be a family, let $\mathcal{I}$ be its extension, set $\mathfrak{i}_{\vec x}\colon\LT^{\T}(\vec x)\to\P(\vec x),
    [\alpha:\vec x]_\T\mapsto\mathcal{I}(\alpha:\vec x)$ and set $\mathbb{J}_n\colon\mathbb{P}_n\to\P(\vec x),\, R \mapsto \mathfrak{i}_{\vec x}([R(\vec x):\vec x]_\T)=\mathcal{I}(R(\vec x):\vec x).$ By \eqref{i:int-subs} in the definition of $\mathcal{I}$, $\mathcal{I}(R(\vec x):\vec x)=\mathbb{I}_n(R)$, so that $\mathbb{J}_n=\mathbb{I}_n$, as desired.
\end{proof}

\begin{proposition}\label{p:quotient-extension}
Let $\T$ be a first-order theory. Quotients of $\LT^{\T}$ in $\DoctABA^{\ctx}$ are in bijection with theories extending $\T$ and closed under deduction.
\end{proposition}
\begin{proof}
    For any first-order Boolean doctrine $\P$ over a small category $\C$, quotients of $\P$ in $\DoctABA^\C$ correspond to congruences on $\P$, which in turn correspond to filters of the fiber over the terminal object $\tmn$, since two formulas $\alpha$ and $\beta$ are in the same congruence class if and only if the universal closure of $\alpha \leftrightarrow \beta$ is in the congruence class of $\top_{\P(\tmn)}$. If $\P=\LT^{\T}$, such filters correspond to deductively closed extensions of $\T$.
\end{proof}

\section{Quantifier-free formulas and quantifier alternation depth for doctrines}
\label{s:q-depth}

Formulas of first-order logic can be stratified by their quantifier alternation depth, which, roughly speaking, corresponds to the number of nested blocks of $\forall$ and $\exists$.

A formula has quantifier alternation depth\dots
\begin{itemize}

    \item
    \dots $=0$ if and only if it is equivalent to a quantifier-free formula.
    
    \item 
    \dots $\leq n+1$ if and only if it is equivalent to a Boolean combination of quantifications of formulas with quantifier alternation depth at most $n$.
\end{itemize}

The notion of quantifier alternation depth can be refined to work modulo a first-order theory $\T$. 
In this case, a formula has quantifier alternation depth modulo $\T$\dots
\begin{itemize}
    \item \dots $=0$ if and only if it is $\T$-equivalent to a quantifier-free formula.

    \item \dots $\leq n+1$ if and only if it is $\T$-equivalent to a Boolean combination of quantifications of formulas with quantifier alternation depth modulo $\T$ at most $n$.
\end{itemize}

\begin{question}\label{question}
    Given a first-order theory $\T$, can we reconstruct the quantifier alternation depth modulo $\T$ from the first-order Boolean syntactic doctrine associated to $\T$? Equivalently, can we recover the notion of being a quantifier-free formula modulo $\T$?
\end{question}

The answer is no, as the following example illustrates.

\begin{example}[Quantifier-freeness modulo a theory is not an intrinsic notion] \label{ex:not-intrinsic}
    Let $\T$ be the theory of preordered sets in the language $\L \coloneqq \{\leq\}$.
    Set $\L'\coloneqq\{\leq, \min\}$, with $\min$ a unary relation symbol, and $\T'\coloneqq\T \cup \{\forall x\, (\min(x) \leftrightarrow \forall y\, (x \leq y))\}$.
    The syntactic doctrines $\LT^{\L,\T}$ and $\LT^{\L',\T'}$ are isomorphic, but the isomorphism does not preserve quantifier-free formulas, as $\forall y\, (x \leq y)$ is quantifier-free modulo $\T'$ but not modulo $\T$.
    Therefore, quantifier-freeness modulo a theory is not an intrinsic notion in the doctrinal setting.
\end{example}

Similar considerations occur, for instance, in the first-order fragment of logic on finite words of automata theory. Here, one distinguishes between the ``pure'' quantifier alternation hierarchy (which corresponds to the Straubing-Thérien hierarchy \cite{Straubing1981,Therien1981}) from other ``enriched'' quantifier alternation hierarchies, where one considers as ``quantifier-free'' (for the purposes of the hierarchy) certain formulas such as 
\begin{align*}
    \max(x) &\coloneqq \lnot \exists y\, x < y,\\
    \min(x) &\coloneqq \lnot \exists y\, y < x,\\
    +1(x, y) &\coloneqq (x < y) \land \lnot \exists z\,(x < z \land  z < y).
\end{align*}
These additions correspond to different levels of the dot-depth hierarchy \cite{CohenBrzozowski1971}. Note that considering $\max(x)$ as quantifier-free is essentially the same as adding a unary relation symbol $M$ to the language and the axiom $\forall x\, (M(x) \leftrightarrow \lnot \exists y\, x < y)$, and then taking the notion of quantifier-freeness modulo the new theory.

\begin{remark} \label{r:comparison}
    We do not know whether quantifier-free formulas become intrinsic when considered modulo a \emph{universal} theory, rather than an arbitrary one (indeed, the latter allows tricks as in \cref{ex:not-intrinsic}). If so, it would be interesting if one could find a characterization analogous to the one relative to the $\{\top, \land, =, \exists\}$-fragment in \cite{MaiettiTrotta2023} (see \cref{q:intrinsic} below).
\end{remark}

We consider a variation of the notion of a first-order Boolean doctrine in which the depth of alternation of quantifiers of formulas is taken into account. 

In this paper, we adopt the convention that $\N$ contains $0$.
Consider the syntactic doctrine of first-order formulas modulo a theory $\T$ as in \cref{fbf}.
For every $n \in \N$ and context $\vec{x}$ we define the Boolean algebra $\LT^\T_n(\vec{x})$ (of ``formulas with quantifier alternation depth at most $n$ modulo $\T$'') inductively on $n$:
\begin{itemize}
    \item 
    we define $\LT^\T_0(\vec x)\subseteq\LT^\T(\vec x)$ as the set of equivalence classes of quantifier-free first-order formulas modulo $\T$ with free variables in $\vec x$;
    
    \item \label{i:generate}
    for $n \geq 0$, $\LT^\T_{n+1}(\vec{x})$ is the Boolean subalgebra of $\LT^\T(\vec{x})$ generated by the elements $\forall\vec{y}\,\alpha(\vec x,\vec y)$ with $\vec{y}$ ranging over contexts and $\alpha(\vec x,\vec y)$ ranging over $\LT^\T_n(\vec x,\vec y)$. 
\end{itemize}

Moreover, we have a chain of inclusions
\[
    \LT^\T_0(\vec{x}) \subseteq \LT^\T_1(\vec{x}) \subseteq \LT^\T_2(\vec{x}) \subseteq \dots
\]
because $\forall\vec{y}\,\alpha(\vec x,\vec y)$ is equivalent to $\alpha(\vec x)$ whenever $\vec y$ is the empty list.

Furthermore, every first-order formula $\alpha(\vec{x})$ belongs to $\LT^\T_{n}(\vec{x})$ for some $n \in \N$. (This can be proved by induction on the complexity of $\alpha(\vec{x})$.)
In other words, 
\[
    \LT^\T(\vec{x}) = \bigcup_{n \in \N} \LT^\T_n(\vec{x}).
\]

The \emph{quantifier alternation depth modulo $\T$} of $\alpha(\vec x) \in \LT^\T(\vec x)$ is the least $n$ such that $\alpha(\vec x)\in \LT^\T_n(\vec x)$.

For each $n\in\N$, the assignment $\LT^\T_n$ can be extended to a functor
\begin{equation} \label{eq:ltn}
    \LT^\T_n\colon\ctx\op\to \BA   
\end{equation}
by defining the reindexing $\LT^\T_n(\vec t(\vec x))\colon \LT^\T_n(\vec{y})\to\LT^\T_n(\vec{x})$ along the tuple of terms $\vec{t}(\vec{x})\colon \vec x\to \vec{y}$ as the restriction of $\LT^\T(\vec t(\vec x))$.

The stratification of $\LT^\T$ into the sequence $\LT^\T_0$, $\LT^\T_1$, \ldots is the motivating example for the definitions that follow.
Since the notion of a first-order Boolean doctrine loses information about the quantifier alternation depth of a formula (\cref{ex:not-intrinsic}), we add further structure to take it into account. We propose two definitions, which we will prove to carry the same information.
The most intrinsic one is the second one.
\begin{enumerate}
    \item In \cref{d:whole-and-first-layer} we define a \emph{quantifier-free fragment} of a first-order Boolean doctrine. A quantifier-free fragment consists of the class of formulas ``considered as quantifier-free''. 
    
    \item In \cref{d:only-layers} we define a \emph{QA-stratified Boolean doctrine}: it consists of a sequence of Boolean doctrines modeling the stratification of the formulas by quantifier-alternation depth.
    An ``ambient'' first-order Boolean doctrine can be obtained as the directed union of all the layers.
\end{enumerate}

We will prove these notions to be in a one-to-one correspondence (\cref{p:equivalent,p:equivalent-cat}).
Moreover, in \cref{t:it-all-makes-sense,t:QA-makes-sense}, we will show that the notions of quantifier-free fragment and QA-stratified Boolean doctrines axiomatize, respectively, the set of quantifier-free formulas modulo a theory and the stratification of formulas under quantifier-alternation depth.

\subsection{Quantifier-free fragments} \label{subs:quantifier-free-fragments}

\begin{definition}[Quantifier-free fragment] \label{d:whole-and-first-layer}
    Let $\P\colon\C\op\to \BA$ be a first-order Boolean doctrine.
    A \emph{quantifier-free fragment of $\P$} is a functor $\P_0 \colon \C\op \to  \BA$ with the following properties.
    \begin{enumerate}
        \item (Subfunctor) \label{i:qff1} $\P_0$ is a subfunctor of $\P$, i.e.\:
        \begin{enumerate}[label = (\roman*), ref = \roman*]
            \item 
            For every $X \in \C$, $\P_0(X)$ is a Boolean subalgebra of $\P(X)$. 
            
            \item
            For every morphism $f \colon X' \to X$ in $\C$, the function $\P(f) \colon \P(X) \to \P(X')$ restricts to the function $\P_0(f)\colon \P_0(X) \to \P_0(X')$.
        \end{enumerate}
        
        \item (Generation) \label{i:qff3}
        For each object $X$ in $\C$, $\P(X) = \bigcup_{n \in \N} \P_n(X)$, where $\P_n(X)$ is the Boolean subalgebra of $\P(X)$ defined inductively on $n$ as follows.
        The poset $\P_0(X)$ is already defined; for $n \geq 0$, $\P_{n+1}(X)$ is the Boolean subalgebra of $\P(X)$ generated by the union of the images of $\P_{n}(X \times Y)$ under $\fa{Y}{X} \colon \P(X \times Y) \to \P(X)$, for $Y$ ranging over the objects of $\C$.  
    \end{enumerate}
\end{definition}

\begin{remark}\label{r:incl}
    In \cref{d:whole-and-first-layer}\eqref{i:qff3}, for every $n \geq 0$, $\P_{n}(X) \subseteq \P_{n+1}(X)$.
    Indeed, $X$ is a particular product of $X$ with the terminal object $\tmn$, and the first projection $\pr_1 \colon X = X \times \tmn \to X$ is the identity.
    The function $\P(\pr_1) \colon \P(X) \to \P(X)$ is the identity, and thus its right adjoint $\fa{\tmn}{X} \colon \P(X) \to \P(X)$ is also the identity.
    Thus, the image of $\P_{n}(X)$ under $\fa{\tmn}{X}$ is $\P_{n}(X)$, which is then contained in $\P_{n+1}(X)$.
\end{remark}

Let $\C$ be a category with finite products. We recall from \cref{r:many-sorted} the presentation of $\DoctABA^\C$ as a variety of many-sorted algebras.
We partition the set of operations into the set $\mathcal{S}$ of substitutions, the set $\mathcal{B}$ of Boolean operations, and the set $\mathcal{Q}$ of quantifications.
Given a set $\mathcal{O}$ of algebraic operations on an algebra $A$, we write $\mathcal{O}(S)$ for the closure of $S\subseteq A$ under $\mathcal{O}$.

With this notation, the conditions in \cref{d:whole-and-first-layer} mean respectively
\begin{enumerate}
    \item $\P_0 = \mathcal{B}(\P_0) = \mathcal{S}(\P_0)$.
    \item $\P = \bigcup_{n \in \N}(\mathcal{B}\circ\mathcal{Q})^{n}(\P_0)$.
\end{enumerate}

We call \emph{QA-stratification} (short for \emph{quantifier alternation stratification}) of a first-order Boolean doctrine $\P$ the sequence of Boolean subfunctors $(\P_n)_{n \in \N}$ of $\P$ obtained by setting
\[
\P_n\coloneqq(\mathcal{B}\circ\mathcal{Q})^{n}(\P_0)
\]
for some quantifier-free fragment $\P_0$ of $\P$.

\begin{example}(Motivating example) \label{ex:motivating-strat-fragment}
    Consider the syntactic doctrine $\LT^\T$ of formulas modulo $\T$ (\cref{fbf}), and its subfunctors $\LT^\T_n$ (for $n \in \N$) as in \eqref{eq:ltn}.
    The functor $\LT^\T_0$ is a quantifier-free fragment of $\LT^\T$, and the sequence $(\LT^\T_n)_{n\in\N}$ is the QA-stratification of $\LT^\T$ arising from $\LT^\T_0$.
\end{example}

\begin{remark}[Quantifier-free fragment = generating Boolean subdoctrine]    
We note that a quantifier-free fragment is simply a \emph{generating Boolean subdoctrine}.
By this we mean that a quantifier-free fragment of a first-order Boolean doctrine $\P$ (\cref{d:whole-and-first-layer}) is precisely an $(\S\cup\B)$-subalgebra $\P_0$ of $\P$ that $(\S \cup \B \cup \Q)$-generates $\P$, i.e.\
\begin{enumerate}
    \item $\P_0 = \mathcal{B}(\P_0) = \mathcal{S}(\P_0)$.
    \item $\P = (\S \cup \B \cup \Q)(\P_0)$.
\end{enumerate}
This relies on the following facts:
\begin{enumerate}
    \item 
    Suppose that the set of algebraic operations decomposes as a union $\mathcal{K}\cup \mathcal{H}$.
    Then any generating $\mathcal{K}$-subalgebra $G$ of a $(\mathcal{K}\cup \mathcal{H})$-algebra $A$ extends to
    a sequence
    \[
    G\subseteq (\mathcal{K}\circ \mathcal{H})(G)\subseteq (\mathcal{K} \circ \mathcal{H})^2(G)\subseteq \dots
    \]
    whose union is $A$.
    In our case, $\mathcal{K}=\mathcal{S}\cup \mathcal{B}$ and $\mathcal{H}=\mathcal{Q}$.

    \item if $\R$ is an $\S$-subalgebra of $\P$, then $((\S \cup \B)\circ\Q)(\R) = (\B\circ\Q)(\R)$: the inclusion $\supseteq$ is obvious, while the inclusion $\subseteq$ holds because $\B\Q(\R)$ is closed under $\S$, indeed
    \begin{align*}
        \S\B\Q(\R) & \subseteq \B\S\Q(\R) &&\text{(by functoriality)}\\
        & \subseteq \B\Q\S(\R) && \text{(by Beck-Chevalley)}\\
        & = \B\Q(\R) &&\text{(since $\R = \S(\R)$)}.
    \end{align*}
\end{enumerate}
\end{remark}

We note that any functor $\P_n$ in a QA-stratification $(\P_n)_{n \in \N}$ of $\P$ is a quantifier-free fragment of $\P$; for example, any $\LT^\T_n$ is a quantifier-free fragment of $\LT^\T$.

As an additional example, we mention that any first-order Boolean doctrine is a quantifier-free fragment of itself.
The rationale behind this is that all first-order formulas can be seen as quantifier-free modulo a new theory in an expanded language; this process is called Morleyzation, and consists in adding to the language a new symbol $R_\alpha$ for each first-order formula $\alpha(\vec{x})$, together with the axiom $\forall\vec{x}\,(R_\alpha(\vec{x}) \leftrightarrow \alpha(\vec{x}))$.

For the interested reader, in \cref{s:no-smallest} we exhibit a first-order Boolean doctrine with no least quantifier-free fragment. In fact, the lack of a least generating subset (or of a least generating subreduct) is a common phenomenon for algebraic structures.

\begin{example}[$\mathscr{P}$ is the unique quantifier-free fragment of itself]
    We show that the subset doctrine $\mathscr{P} \colon \Set\op \to \BA$ has a unique quantifier-free fragment: itself.
    That the whole $\mathscr{P}$ is a quantifier-free fragment of itself is easily seen, as this holds for any first-order Boolean doctrine, as we just saw.
    
    To prove that this is the only one, let $\P_0$ be a quantifier-free fragment of $\mathscr{P}$.
    We first show that $\P_0$ is not the subfunctor $\mathbf{T} \colon \Set\op \to \BA$ of $\mathscr{P}$ defined by $\mathbf{T}(X) \coloneqq \{\varnothing, X\}$.
    To do so, we prove that $\mathbf{T}$ is closed under quantifiers. (From this it will follow that $\mathbf{T}$ is not a quantifier-free fragment of $\mathscr{P}$, since $\mathbf{T} \neq \mathscr{P}$.) Recall the interpretation of the universal quantifier in $\mathscr{P}$:
    \begin{align*}
    \fa{Y}{X}\colon \mathscr{P}(X\times Y)&\longrightarrow\mathscr{P}(X)\\
    S&\longmapsto\{x\in X\mid  \text{for all }y \in Y, \, (x, y) \in S\}.
    \end{align*}
    The image of $\top_{\mathscr{P}(X\times Y)}=X\times Y$ is $\top_{\mathscr{P}(X)}=X$. The image of $\bot_{\mathscr{P}(X\times Y)}=\varnothing$ is $\bot_{\mathscr{P}(X)}=\varnothing$ if $Y\neq\varnothing$, and $X$ otherwise.
    This shows that $\mathbf{T}$ does not ``generate'' $\mathscr{P}$ and hence is not a quantifier-free fragment of it. Therefore, $\P_0 \neq \mathbf{T}$.
    
    From $\P_0 \neq \mathbf{T}$ we deduce that $\P_0$ strictly contains $\mathbf{T}$. 
    Thus, there is $X \in \Set$ such that $\P_0(X) \neq \{\varnothing, X\}$.
    Therefore, there is $A \in \P_0(X)$ such that $\varnothing \ne A \neq X$, and so there are $a \in A$ and $b \in X \setminus A$.
    Let $Y \in \Set$, and let us show that $\P_0(Y) = \mathscr{P}(Y)$.
    Let $C \in \mathscr{P}(Y)$.
    Let $f \colon Y \to X$ be the function that maps an element $y \in Y$ to $a$ if $y \in C$ and to $b$ otherwise.
    Since $\P_0$ is closed under reindexings, from $A \in \P_0(X)$ it follows that the element $C = f^{-1}[A] = \mathscr{P}(f)(A)$ of $\mathscr{P}(Y)$ belongs to $\P_0(Y)$.
    Since $C$ is arbitrary, $\P_0(Y) = \mathscr{P}(Y)$.
\end{example}

As stated in \cref{ex:motivating-strat-fragment}, the set of quantifier-free formulas modulo a theory $\T$ form a quantifier-free fragment of the syntactic doctrine $\LT^\T$ in the sense of \cref{d:whole-and-first-layer}. We now show a converse below: every quantifier-free fragment of a first-order Boolean doctrine $\P$ over a category of contexts arises as the image under an isomorphism $\LT^\T\to\P$ of the quantifier-free fragment of quantifier-free formulas of a syntactic doctrine $\LT^\T$.

\begin{theorem}[A quantifier-free fragment is the set of quantifier-free formulas modulo a theory] \label{t:it-all-makes-sense}
    Let $\P_0$ be a quantifier-free fragment of a first-order Boolean doctrine $\P\colon\ctx\op\to\BA$, where $\ctx$ is the category of contexts for a functional language $\mathbb{F}$.
    There are a relational language $\mathbb{P}$, a theory $\mathcal{T}$ in the language $(\mathbb{F},\mathbb{P})$ and a first-order Boolean doctrine isomorphism between $\P$ and $\LT^{\T}$ under which $\P_0$ corresponds to $\LT_0^{\mathcal{T}}$ (= the subfunctor of $\LT^{\mathcal{T}}$ of quantifier-free formulas modulo $\mathcal{T}$). 
\end{theorem}

\begin{proof}
    For each $n\in\N$, we set $\mathbb{P}_n \coloneqq \{R_{\gamma}\mid {\gamma\in\P_0(x_1,\dots,x_n)}\}$ as the set of $n$-ary predicate symbols, $\mathbb{P}\coloneqq\bigcup_{n\in\N}\mathbb{P}_n$, and $\L\coloneqq(\mathbb{F},\mathbb{P})$. 
    We define the family of functions $(\mathbb{I}_n\colon\mathbb{P}_n\to\P,\, R_\gamma\mapsto\gamma)_{n \in \N}$, and we let $\mathcal{I}\colon\mathrm{Form}(\L)\to\bigcup_{\vec x\in\ctx}\P(\vec x)$ be its extension as in \cref{r:int-fmlas}. Define $\mathcal{T}\coloneqq\{\alpha\mid\mathcal{I}(\alpha:())=\top_{\P()}\}$. 
    
    Let $(\id,\mathfrak{i})\colon\LT^{\T}\to\P$ be the first-order Boolean doctrine morphism associated to the family $(\mathbb{I}_n)_{n}$ by the bijection in \cref{p:univ-prop-ltlt}; for all $\vec x\in\ctx$, $\mathfrak{i}_{\vec x}([\alpha:\vec x]_\T)=\mathcal{I}(\alpha:\vec x)$.
    We have that $\mathfrak{i}$ is componentwise surjective since its image is the subalgebra of $\P$ generated by the images of $(\mathbb{I}_n)_n$, i.e.\ the subalgebra of $\P$ generated by $\P_0$, i.e.\ $\P$.
    It is easy to prove that componentwise injectivity of $\mathfrak{i}$ is equivalent to
    \[\mathfrak{i}_{()}([\alpha:()]_\T)=\top_{\P()}\iff[\alpha:()]_\T=[\top:()]_\T,\]
    and the latter follows from the definition of $\T$.
    
    We now prove that, under the isomorphism $(\id,\mathfrak{i})\colon \LT^{\T}\to\P$, the formulas in $\LT^{\T}$ that are quantifier-free modulo $\T$ correspond to the elements in $\P_0$.

    \[
    \begin{tikzcd}
    	{\LT^{\T}_0} & {\LT^{\T}}  \\
    	{\P_0} & {\P}
    	\arrow[hook, from=1-1, to=1-2]
    	\arrow[dashed, from=1-1, to=2-1]
    	\arrow["{\text{\rotatebox[origin=c]{90}{$\sim$}}}"',"{(\id,\mathfrak{i})}", from=1-2, to=2-2]
    	\arrow[hook, from=2-1, to=2-2]
    \end{tikzcd}
    \]

    For all $\vec x\in\ctx$ and for all $\gamma\in\P(\vec x)$ we have \[\mathfrak{i}_{\vec x}([R_\gamma(\vec x):\vec x]_\T)=\mathcal{I}(R_\gamma(\vec x):\vec x)=\mathbb{I}_n(R_\gamma)=\gamma.\]
    
    Since $\LT^{\T}_0$ is the Boolean doctrine generated by $\{[R_\gamma(\vec x):\vec x]_\T\}_{\vec x\in\ctx, \gamma\in\P_0(\vec x)}$, its image in $\P$ is the Boolean doctrine generated by $\{\gamma\}_{\vec x\in\ctx, \gamma\in\P_0(\vec x)}$, which is $\P_0$, since it is closed under Boolean operations and substitutions. Conversely, for every $\gamma\in\P_0(\vec x)$, we have $\mathfrak{i}_{\vec x}^{-1}(\gamma)=[R_\gamma(\vec x):\vec x]_\T\in\LT^{\T}_0$. 
\end{proof}

We observe that, taking $\P_0 =\P$ in \cref{t:it-all-makes-sense}, we get the following statement, which says that any first-order Boolean doctrine over the category of contexts for some functional language is a syntactic doctrine.

\begin{corollary}
    Let $\P$ be a first-order Boolean doctrine over the category of contexts for a functional language $\mathbb{F}$. There is a set $\mathbb{P}$ of predicate symbols and a theory $\T$ in the language $(\mathbb{F},\mathbb{P})$ with $\P\cong\LT^{\T}$.
\end{corollary}

\subsection{QA-stratified Boolean doctrines}

The notion of quantifier-free fragment (and of its corresponding QA-stratification) in \cref{d:whole-and-first-layer} requires the ``ambient'' first-order Boolean doctrine to be given.
We will soon define QA-stratified Boolean doctrines as certain sequences $(\P_n)_n$ of Boolean doctrines, without an ambient first-order Boolean doctrine given a priori.
But first, we single out the properties that the pairs $(\P_n, \P_{n+1})$ should satisfy (in the style of \cite{CoumansVanGool2013}).

\begin{definition}[QA-one-step Boolean doctrine] \label{d:one-step-doctrine}
    A \emph{QA-one-step Boolean doctrine} over a category $\C$ with finite products is a componentwise injective natural transformation $i \colon \P_0 \hookrightarrow \P_1$ between two functors $\P_0, \P_1 \colon \C\op \to  \BA$ with the following properties.
    \begin{enumerate}
        \item \label{i:one-step-universal}
        (One-step universal)
        For every projection $\pr_1 \colon X \times Y \to X$ in $\C$ and $\beta \in \P_0(X \times Y)$, there is an element $\fa{Y}{X} \beta \in \P_{1}(X)$ such that, for all $\alpha \in \P_{1}(X)$, 
       \[
           \alpha \leq \fa{Y}{X} \beta \text{ in $\P_{1}(X)$} \Longleftrightarrow \P_{1}(\pr_1)(\alpha) \leq i_{X \times Y}(\beta) \text{ in $\P_{1}(X \times Y)$}.
       \]
       (Note that one such element $\fa{Y}{X} \beta$ is unique.)
        
        \item \label{i:one-step-Beck-Chevalley}
        (One-step Beck-Chevalley)
        For every morphism $f\colon X'\to X$ in $\C$, the following diagram in $\Pos$ commutes.
        \[
            \begin{tikzcd}
                {\P_0(X\times Y)} & {\P_{1}(X)}  \\
                {\P_0(X'\times Y)} & {\P_{1}(X')} 
                \arrow["{\P_0(f\times\id_{Y})}", from=1-1, to=2-1, swap]
                \arrow["{\P_{1}(f)}", from=1-2, to=2-2]
                \arrow["{\fa{Y}{X'}}"', from=2-1, to=2-2]
                \arrow["{\fa{Y}{X}}", from=1-1, to=1-2]
            \end{tikzcd}
        \]
        \item (One-step generation)
        For all $X \in \C$, $\P_{1}(X)$ is generated, as a Boolean algebra, by the union of the images of the functions $\fa{Y}{X} \colon \P_0(X \times Y) \to \P_{1}(X)$ for $Y$ ranging over $\C$.
    \end{enumerate}

\end{definition}

Since the natural transformation $i \colon \P_0 \hookrightarrow \P_1$ is componentwise injective, we will often suppose $i$ to be a componentwise inclusion, avoiding mentioning $i$, or thinking of $i \colon \P_0 \hookrightarrow \P_1$ as a pair $(\P_0, \P_1)$.

\begin{definition}[QA-stratified Boolean doctrine] \label{d:only-layers}
    A \emph{QA-stratified Boolean doctrine} (short for \emph{quantifier alternation stratified Boolean doctrine}) over a category $\C$ with finite products is a sequence of functors $(\P_n \colon \C\op \to  \BA)_{n \in \N}$ and QA-one-step Boolean doctrines $(i_n \colon \P_n \hookrightarrow \P_{n+1})_{n \in \N}$ such that, for all $X,Y\in \C$ and $n\in \N$, the following diagram commutes:
    \[
        \begin{tikzcd}[column sep = 4 em]
            \P_n(X \times Y) \arrow{r}{\fa{Y}{X,n}} \arrow[hook,swap]{d}{(i_n)_{X \times Y}} & \P_{n+1}(X)\arrow[hook]{d}{(i_{n+1})_X}\\
            \P_{n + 1}(X \times Y) \arrow[swap]{r}{\fa{Y}{X,n+1}} & \P_{n+2}(X)
        \end{tikzcd}
    \]
    where, for every $n\in\N$, $\fa{Y}{X,n}\colon\P_n(X\times Y)\to\P_{n+1}(X)$ is the one-step quantifier of the QA-one-step Boolean doctrine $(\P_n,\P_{n+1})$.
\end{definition}

As above, since for every $n\in\N$ the natural transformation $i_n \colon \P_n \hookrightarrow \P_{n+1}$ is componentwise injective, we will often suppose $i_n$ to be the componentwise inclusion of fibers of $\P_n$ into fibers of $\P_{n+1}$, thinking of $((\P_n)_{n\in\N},(i_n)_{n\in\N})$ simply as a sequence $(\P_n)_{n\in\N}$.

The fact that the axiomatization of QA-stratified Boolean doctrines involves only pairs and not triples $\P_n \hookrightarrow \P_{n+1} \hookrightarrow \P_{n+2}$  (besides the connecting condition saying that each pair extends the previous one) is due to the fact that the equations for quantifiers (adjunction and Beck–Chevalley) are of rank 1, meaning that they only involve nesting one quantifier.

\begin{proposition} \label{p:equivalent}
    QA-stratified Boolean doctrines are in 1:1 correspondence with first-order Boolean doctrines equipped with a quantifier-free fragment of it.\footnote{Here, we are supposing that the componentwise injective natural transformations in QA-stratified Boolean doctrines are subset inclusions.}
\end{proposition}

\begin{proof}
    Let $\P_0$ be a quantifier-free fragment of a first-order Boolean doctrine $\P$, and $(\P_n)_{n\in\N}$ the induced QA-stratification.
    The sequence $\P_0 \hookrightarrow \P_1 \hookrightarrow \P_2\hookrightarrow \dots$ is a QA-stratified Boolean doctrine; the proof of this is straightforward once one defines the ``partial'' quantifiers as the restrictions of the quantifiers in $\P$. 

    Conversely, given a QA-stratified Boolean doctrine $((\P_n \colon \C\op \to  \BA)_{n \in \N},(i_n)_{n\in\N})$, we obtain a Boolean doctrine as the (fiberwise computed) colimit $\P$ of $(\P_n \colon \C\op \to  \BA)_{n \in \N}$, as follows.
    For $X \in \C$, we set $\P(X) = \bigcup_{n \in \N} \P_n(X)$ (where we are assuming that the comparison natural transformations $i_n$ are componentwise subset inclusions), with the obvious structure of a Boolean algebra.
    For a morphism $f \colon X' \to X$ in $\C$, the function $\P(f) \colon \P(X) \to \P(X')$ is the obvious one.
    We first prove that $\P$ is a first-order Boolean doctrine. For every projection $\pr_1 \colon X \times Y\to X$, define the function $\fa{Y}{X}\colon\P(X \times Y)\to\P(X)$ as $\fa{Y}{X} \coloneqq \bigcup_{n\in\N}\fa{Y}{X,n}$. This is well-defined by the commutativity of the diagram in \cref{d:only-layers}. To check that $\fa{Y}{X}$ is the right adjoint to $\P(\pr_1)$, let $\alpha\in \P(X)$ and $\beta\in\P(X \times Y)$.
    There is $n\in\N$ large enough so that $\alpha\in \P_{n+1}(X)$ and $\beta\in\P_n(X \times Y)$.
    Then,
    \begin{align*}
        &\alpha\leq\fa{Y}{X}\beta && \text{(in $\P(X)$)}\\
        &\iff \alpha\leq \fa{Y}{X,n}\beta && \text{(in $\P_{n+1}(X)$)}\\
        &\iff \P_{n+1}(\pr_1)(\alpha)\leq i_{X \times Y, n}(\beta) &&\text{(in $\P_{n+1}(X \times Y)$)}\\
        &\iff\P(\pr_1)(\alpha)\leq\beta &&\text{(in $\P(X \times Y)$)}.
    \end{align*}
    Hence, $\fa{Y}{X}$ is the right adjoint of $\P(\pr_1)$. 
    The Beck-Chevalley condition follows from \cref{d:one-step-doctrine}\eqref{i:one-step-Beck-Chevalley}. 
    So, $\P$ is indeed a first-order Boolean doctrine, and $\P_0$ is easily seen to be a quantifier-free fragment of $\P$.

    Finally, it is immediate that these assignments are mutually inverse.
\end{proof}

The correspondence between QA-stratified Boolean doctrines and first-order Boolean doctrines equipped with a quantifier-free fragment can be made into a categorical equivalence; we introduce the two categories.

\begin{definition}[The category $\QFF$ of quantifier-free fragments]
    Let $\P$ and $\R$ be first-order Boolean doctrines, and let $\P_0$ and $\R_0$ be quantifier-free fragments of $\P$ and $\R$, respectively.
    A \emph{morphism of quantifier-free fragments} from $(\P_0, \P)$ to $(\R_0, \R)$ is a first-order Boolean doctrine morphism $(M, \m) \colon \P \to \R$ that restricts to a Boolean doctrine morphism $\P_0 \to \R_0$.
    \[\begin{tikzcd}
        \P_0\arrow[hook,r]\arrow[d,dashed] & \P\arrow[d,"{(M,\m)}"] \\
        \R_0\arrow[hook,r] & \R
    \end{tikzcd}\]

    We let $\QFF$ denote the category whose objects are pairs $(\P_0, \P)$ with $\P_0$ a quantifier-free fragment of a first-order Boolean doctrine $\P$ and whose morphisms are morphisms of quantifier-free fragments.
\end{definition}

\begin{definition}[Morphism of QA-one-step Boolean doctrines]
    Let $\P_0 \hookrightarrow \P_1$ and $\R_0 \hookrightarrow \R_1$ be QA-one-step Boolean doctrines over $\C$ and $\D$, respectively.
    A \emph{morphism of QA-one-step Boolean doctrines} from $(\P_0, \P_1)$ to $(\R_0, \R_1)$ is a triple $(M, j_0,j_1)$ with $M \colon \C \to \D$ a functor preserving finite products and $j_0 \colon \P_0 \to \R_0 \circ M\op$ and $j_1 \colon \P_1 \to \R_1 \circ M\op$ natural transformations making the following diagram commute
    \[
    \begin{tikzcd}
        \P_0 \arrow[hook]{r}{} \arrow[swap]{d}{j_0} & \P_1 \arrow{d}{j_1}\\
        \R_0\circ M\op \arrow[swap,hook]{r}{} & \R_1\circ M\op
    \end{tikzcd}
    \]
    and such that, for all $X, Y \in \C$, the following diagram in $\Pos$ commutes.
    \[
        \begin{tikzcd}[column sep = 5em]
            \P_0(X \times Y) \arrow{r}{\fa{Y}{X}} \arrow[swap]{d}{(j_0)_{X \times Y}} & \P_{1}(X)\arrow{d}{(j_1)_{X}}\\
            \R_{0}(M(X) \times M(Y)) \arrow[swap]{r}{\fa{M(Y)}{M(X)}} & \R_{1}(M(X))
        \end{tikzcd}
    \]
\end{definition}

\begin{definition}[The category $\QA$ of QA-stratified Boolean doctrines]
    A \emph{morphism of QA-stratified Boolean doctrines} from $(\P_n)_{n \in \N}$ to $(\R_n )_{n \in \N}$ is a pair $(M, (\m_n)_{n\in\N})$ such that, for every $n\in\N$, $(M,\m_n,\m_{n+1})$ is a homomorphism of one-step  Boolean doctrines from $(\P_n, \P_{n+1})$ to $(\R_{n}, \R_{n+1})$.
    
    We let $\QA$ denote the category whose objects are QA-stratified Boolean doctrines and whose morphisms are morphisms of QA-stratified Boolean doctrines.
\end{definition}

\begin{proposition}\label{p:equivalent-cat}
     The categories $\QFF$ and $\QA$ are equivalent. 
\end{proposition}
\begin{proof}
    The functor $U\colon\QFF\to\QA$ is defined on objects as in the proof of \cref{p:equivalent} (it defines the sequence and forgets the ambient doctrine), and maps a morphism $(M,\m)\colon (\P_0,\P)\to (\R_0,\R)$ to the pair $(M,(\m_{\mid\P_n})_{n\in\N})$, where, for every $n\in\N$, the natural transformation $\m_{\mid\P_n}\colon \P_n\to \R_n$ is componentwise the restriction of $\m$. It is easy to see that $U$ is well-defined. The functor $L\colon\QA\to \QFF$ is defined on objects as in the proof of \cref{p:equivalent} (it keeps the first entry of the sequence, and defines the ambient doctrine by taking the union of all the layers). The functor $L$ maps a morphism $(M, (\m_n)_{n\in\N})\colon (\P_n)_{n\in\N}\to (\R_n)_{n\in\N}$ to the pair $(M,\bigcup_{n\in\N}\m_{n})$.
    It is easy to see that $L$ is well-defined.
    \[\begin{tikzcd}[row sep=scriptsize, column sep=3.5em]
    	{\QFF} & \QA && \QA & {\QFF} \\
    	{(\P_0,\P)} & {(\P_n)_{n\in\N}} && {(\P_n)_{n\in\N}} & {(\P_0,\bigcup_{n\in\N}\P_n)} \\
    	\\
    	{(\R_0,\R)} & {(\R_n)_{n\in\N}} && {(\R_n)_{n\in\N}} & {(\R_0,\bigcup_{n\in\N}\R_n)}
    	\arrow["U", from=1-1, to=1-2]
    	\arrow["L", from=1-4, to=1-5]
    	\arrow[""{name=0, anchor=center, inner sep=0}, "{(M,\m)}"', from=2-1, to=4-1]
    	\arrow[""{name=1, anchor=center, inner sep=0}, "{(M, (\m_{\mid\P_n})_{n\in\N})}", from=2-2, to=4-2]
    	\arrow[""{name=2, anchor=center, inner sep=0}, "{(M, (\m_n)_{n\in\N})}"', from=2-4, to=4-4]
    	\arrow[""{name=3, anchor=center, inner sep=0}, "{(M,\bigcup_{n\in\N}\m_n)}", from=2-5, to=4-5]
    	\arrow[shorten <=15pt, shorten >=15pt, maps to, from=0, to=1]
    	\arrow[shorten <=15pt, shorten >=15pt, maps to, from=2, to=3]
    \end{tikzcd}\]
    By \cref{p:equivalent}, the composites $U\circ L$ and $L\circ U$ are the identities on objects, and it is easily seen that both composites are the identity on morphisms. Thus, $U$ and $L$ are mutually inverse functors.
\end{proof}

For any theory $\T$, the sequence $(\LT_n^{\T})_n$  is a QA-stratified Boolean doctrine.  
The converse holds, too: every QA-stratified Boolean doctrine over a category of contexts arises in this way: 

\begin{theorem}[Any QA-stratified Boolean doctrine is the stratification by QA-depth modulo some theory]\label{t:QA-makes-sense}
    Let $(\P_n\colon\ctx\op\to\BA)_{n\in\N}$ be a QA-stratified Boolean doctrine, where $\ctx$ is the category of contexts for a functional language $\mathbb{F}$.
    There are a relational language $\mathbb{P}$, a theory $\mathcal{T}$ in the language $(\mathbb{F},\mathbb{P})$ and an 
    isomorphism of QA-stratified Boolean doctrines between $(\P_n)_{n\in\N}$ and $(\LT_n^{\T})_{n\in\N}$.
\end{theorem}

\begin{proof}
    By \cref{t:it-all-makes-sense} and \cref{p:equivalent-cat}.
\end{proof}

\section{Boolean doctrines characterize quantifier-free fragments}\label{sec:char-qff}

In the previous section we captured sets of quantifier-free formulas modulo a theory via the notion of quantifier-free fragment.
In this section we answer the question: what is the intrinsic structure of a quantifier-free fragment?
In other words, what are the structures occurring as quantifier-free fragments of some first-order Boolean doctrine?

We prove that, under a smallness condition, these are precisely the Boolean doctrines: by definition, a quantifier-free fragment of a first-order Boolean doctrine is a Boolean doctrine; to prove the converse direction, we show that every Boolean doctrine $\P$ over a small base category embeds into a first-order Boolean doctrine $\H$. Then, we observe that $\P$ is a quantifier-free fragment of the first-order Boolean doctrine generated by the image of $\P$ in $\H$.

\begin{notation}
    Given a Boolean doctrine $\P \colon \C\op \to \BA$, $X \in \C$ and a complete Boolean algebra $B$, we denote by $\H_{X}^B \colon \C\op \to \BA$ the first-order Boolean doctrine that maps an object $Y \in \C$ to $B^{\Hom(X,Y)}$ and maps a morphism $f \colon Z \to Y$ to the function
    \begin{align*}
        B^{\Hom(X,Y)} & \longrightarrow B^{\Hom(X,Z)}\\
        h & \longmapsto h(f \circ -).
    \end{align*}
    Here, the universal and existential quantifiers are as follows: for $Y,Z\in\C$ and $g\in B^{\Hom(X,Y\times Z)}$ we have

    \begin{align*}
    \fa{Z}{Y}(g)\colon\Hom(X,Y)&\longrightarrow B\\ 
    f&\longmapsto\bigwedge_{h\colon X\to Z}g(\ple{f,h})
    \end{align*}
    \begin{align*}
    \ex{Z}{Y}(g)\colon\Hom(X,Y)&\longrightarrow B\\ 
    f & \longmapsto \bigvee_{h\colon X\to Z}g(\ple{f,h})
    \end{align*}
    (see e.g.\ \cite[Ex.~2.5(d)]{MaiPaRo} for the existential quantifier).
\end{notation}

Before starting the proof, we observe that, given a Boolean doctrine $\P$ over $\C$, $X \in \C$, a complete Boolean algebra $B$ and a Boolean homomorphism $f \colon \P(X) \to B$, we have a morphism $\m$ in $\DoctBA^\C$ from $\P$ to $\H^B_X$ whose component at $Y\in\C$ is defined as follows:
\begin{align}\label{eq:complete-ba}
        \m_Y \colon \P(Y) & \longrightarrow \H_X^B(Y) = B^{\Hom(X,Y)}\\
        \gamma & \longmapsto f(\P(-)(\gamma)).\notag
    \end{align}
Each component of $\m$ preserves Boolean operations since both $f$ and reindexings of $\P$ do.

\begin{theorem} \label{t:bd-embeds-fo}
    Every Boolean doctrine over a small category embeds into a first-order Boolean doctrine.
\end{theorem}

\begin{proof}\footnote{We are grateful to the referee for having suggested this constructive proof.}
    Every Boolean algebra embeds into some complete Boolean algebra; for example, its canonical extension \cite{Gehrke2018}.
    For each $X \in \C$, we let $j^X\colon \P(X) \hookrightarrow \P(X) ^\delta$ denote the canonical extension of $\P(X)$.
   
    Using the smallness of $\C$ and the fact that $\DoctABA^\C$ is a variety, we define the first-order Boolean doctrine
    \[
    \H\coloneqq \prod_{X\in\C}\H^{\P(X) ^\delta}_X.
    \]
    We define a morphism $\mathfrak{i}\colon \P\to\H$ in $\DoctBA^\C$ by defining for every $X\in\C$ a morphism $\mathfrak{i}^X\colon \P\to\H^{\P(X) ^\delta}_X$. We do so as in \eqref{eq:complete-ba} above:
    \[(\mathfrak{i}^X)_Y(\gamma)=j^X(\P(-)(\gamma)),\]
    for every $Y\in\C$ and $\gamma\in\P(Y)$.
    
    We are left to prove that $\mathfrak{i}=(\mathfrak{i}^X)_{X\in\C}$ is injective.
    Let $X \in \C$.
    Let $\alpha, \beta \in \P(X)$ and suppose that the images of $\alpha$ and $\beta$ coincide, i.e.\ $\mathfrak{i}_X(\alpha) = \mathfrak{i}_X(\beta)$.
    In particular the $X$-th component of $\mathfrak{i}_X(\alpha)$ and $\mathfrak{i}_X(\beta)$ coincide, i.e.\ $(\mathfrak{i}^X)_X(\alpha) = (\mathfrak{i}^X)_X(\beta)$. Evaluating both of them at $\id_X$, we obtain
    \[j^X(\P(\id_X)(\alpha))=j^X(\P(\id_X)(\beta)),\]
    i.e.\ $j^X(\alpha)=j^X(\beta)$. Since $j^X$ is injective, we obtain $\alpha = \beta$, as desired.
\end{proof}

\begin{theorem}\label{t:character-0th-layer}
     Every Boolean doctrine over a small category is a quantifier-free fragment of some first-order Boolean doctrine.
\end{theorem}

\begin{proof}
    Let $\P$ be a Boolean doctrine over a small category.
    By \cref{t:bd-embeds-fo}, there is an embedding $i \colon \P \hookrightarrow \H$ into a first-order Boolean doctrine $\H$.
    Then, $\P$ is a quantifier-free fragment of the first-order Boolean doctrine generated by the image of $i$.
\end{proof}

This shows that, under a smallness assumption, the structures occurring as quantifier-free fragments are precisely the Boolean doctrines.

\section{Quantifier completion of a Boolean doctrine}\label{sec:quantifier compl}

Completions under quantifiers have been treated in various contexts (see, for example, \cite{Hofstra2006,Trotta2020}).
In this section, we show that every Boolean doctrine over a small category has a \emph{quantifier completion} (which freely adds quantifiers, see \cref{d:quantifier-completion}), of which, furthermore, it is a quantifier-free fragment.  

\begin{definition}[Quantifier completion] \label{d:quantifier-completion}
     Let $\P\colon \C\op\to\BA$ be a Boolean doctrine. A \emph{quantifier completion of $\P$} is a Boolean doctrine morphism $(I,\mathfrak{i})\colon \P \to \P^\EA$, where $\P^\EA\colon {\C'}\op\to\BA$ is a first-order Boolean doctrine, with the following universal property: for every first-order Boolean doctrine $\R \colon \D \op \to \BA$ and every Boolean doctrine morphism $(M,\m) \colon \P \to \R$ there is a unique first-order Boolean doctrine morphism $(N,\n) \colon \P^\EA \to \R$ such that $(M,\m) = (N,\n) \circ (I, \mathfrak{i})$.
    \[ 
        \begin{tikzcd}
             \P\arrow[r,"{(I,\mathfrak{i})}"',swap]\arrow[dr,"{(M,\m)}",swap] & \P^\EA\arrow[d,"{(N,\n)}"',dashed,swap]\\& \R
        \end{tikzcd}
    \]
\end{definition}

The idea to prove that every Boolean doctrine over a small category has a quantifier completion is that a first-order Boolean doctrine is defined by equations (see \cref{r:many-sorted}), and classes defined by equations have free algebras.

\begin{theorem} \label{t:left-adjoint}
    For any small category $\C$, the forgetful functor $U \colon \DoctABA^\C \to \DoctBA ^\C$ has a left adjoint. Moreover, the unit $\id_{\DoctBA^\C} \to UF$ is componentwise injective.
\end{theorem}

\begin{proof}
    Any forgetful functor between (possibly many-sorted) varieties has a left adjoint.
    In fact, this holds for any functor between varieties induced by a morphism of theories; see, e.g.\ \cite[Prop.~9.3(2)]{AdamekRosickyEtAl2011}.
    This guarantees that $U$ has a left adjoint $F$.
    The fact that the unit is componentwise injective follows from the fact that every Boolean doctrine over $\C$ embeds into some first-order Boolean doctrine over $\C$ (\cref{t:bd-embeds-fo}).
\end{proof}

For any Boolean doctrine $\P$ over a small category $\C$, the universal property of $UF(\P)\colon \C\op \to \BA$ in \cref{t:left-adjoint} ranges over first-order Boolean doctrines over the same category $\C$.
We now show that this implies the universal property with respect to \emph{any} first-order Boolean doctrine, i.e.\ that $UF(\P) \colon \C\op \to \BA$ is a quantifier completion of $\P$ (\cref{t:quantif-complet} below). The following remark contains the key ingredient.

\begin{remark}[Change of base]\label{r:change-of-base}
    Let $\P,\R$ be Boolean doctrines and $(M,\m) \colon \P \to \R$ a Boolean doctrine morphism. We can factor $(M,\m)$ as the composition of two Boolean doctrine morphisms as follows:
    \[
    \begin{tikzcd}[row sep=30pt]
    	\C\op & \C\op & \D\op \\
    	& \BA.
    	\arrow[""{name=0, anchor=center, inner sep=0}, "\P"', from=1-1, to=2-2]
    	\arrow[""{name=1, anchor=center, inner sep=0}, "\R", from=1-3, to=2-2]
    	\arrow["\id_{\C}\op", from=1-1, to=1-2]
    	\arrow["M\op", from=1-2, to=1-3]
    	\arrow[""{name=2, anchor=center, inner sep=0}, "{\R \circ M\op}"{description}, from=1-2, to=2-2]
    	\arrow["\m", curve={height=-6pt}, shorten <=4pt, shorten >=4pt, from=0, to=2]
    	\arrow["\id", curve={height=-6pt}, shorten <=4pt, shorten >=4pt, from=2, to=1]
    \end{tikzcd}
    \]
    If $\R$ has quantifiers, it is easily seen that the Boolean doctrine $\R \circ M\op$ has quantifiers, too, and that the Boolean doctrine morphism $(M,\id) \colon \R \circ M\op \to \R$ in the factorization $(M,\m) = (M,\id)\circ (\id_{\C},\m)$ preserves them.
    Moreover, if also $\P$ has quantifiers and these are preserved by $(M,\m)$, then $(\id_{\C},\m)$ preserves them.
\end{remark}

\begin{theorem}\label{t:quantif-complet}
    Any Boolean doctrine $\P$ over a small category $\C$ has a quantifier completion, namely $(\id_\C,\eta_\P)\colon \P\to UF(\P)$, where $\eta$ is the (componentwise injective) unit of the adjunction $F\dashv U$ in \cref{t:left-adjoint}.
\end{theorem}

\begin{proof}
    Let $\R \colon \D \op \to \BA$ be a first-order Boolean doctrine and let $(M,\m) \colon \P \to \R$ be a Boolean doctrine morphism.
    We shall prove that there is a unique first-order Boolean doctrine morphism $(N,\n) \colon UF(\P) \to \R$ such that $(M,\m) = (N,\n) \circ (\id_\C, \eta_\P)$.
    \[ 
        \begin{tikzcd}
             \P\arrow[r,"{(\id_\C, \eta_\P)}"',swap]\arrow[dr,"{(M,\m)}",swap] & UF(\P)\arrow[d,"{(N,\n)}"',dashed,swap]\\& \R
        \end{tikzcd}
    \]
    
    By \cref{r:change-of-base}, we can factor $(M,\m) = (M,\id) \circ (\id_{\C},\m)$, where $(\id_{\C},\m) \colon \P \to \R \circ M\op$ is a Boolean doctrine morphism and $(M,\id) \colon \R \circ M\op \to \R$ a first-order Boolean doctrine morphism.
    Note that the composite $\R \circ M\op \colon \C\op \to \BA$ has $\C$ as base category and so it belongs to $\DoctABA^\C$.
    By the universal property of $\eta_\P \colon \P \to UF(\P)$ with respect to the morphism $\m \colon \P \to \R \circ M\op$ in $\DoctBA^\C$, there is a unique morphism $\n \colon UF(\P) \to \R \circ M\op$ in $\DoctABA^\C$.
    Then $(M,\n) = (M,\id) \circ (\id_\C,\n)$ is a first-order Boolean doctrine morphism making the outer triangle below commute.
    \[
    \begin{tikzcd}[row sep=2.25em]
         \P\arrow[r,"{(\id_{\C},\eta_\P)}"',swap]\arrow[dr,"{(\id_{\C},\m)}"{description},swap]\arrow[ddr,bend right,"{(M,\m)}"'] & UF(\P)\arrow[d,"{(\id_{\C},\n)}"',dashed,swap]\\& \R \circ M\op \arrow[d,"{(M,\id)}"]\\& \R
    \end{tikzcd}
    \]
    
    We conclude by proving uniqueness. Let $(N',\n')\colon UF(\P) \to \R$ be a first-order Boolean doctrine morphism such that $(M,\m)=(N',\n')\circ(\id_\C,\eta_\P)$. Observe that $N' = N' \circ \id_\C = M$. Moreover, we can factor $(N',\n') = (M,\n')$ as $(M,\id)\circ(\id_\C,\n')$, so it is enough to prove that $(\id_\C,\n)=(\id_\C,\n')$. By \cref{r:change-of-base}, $(\id_\C,\n')$ is a first-order Boolean doctrine morphism, and thus $\n'\colon UF(\P)\to \R\circ M\op$ is a morphism in $\DoctABA^\C$. By the universal property of $\eta_\P$,
    the equality $\n=\n'$ holds if and only if the equality $\n\circ \eta_\P=\n'\circ\eta_\P$ holds, but the latter follows from the equality $(M,\m) = (M,\n') \circ (\id_\C,\eta_\P)$. 
\end{proof}

\begin{corollary} \label{c:left-adjoint-small}
    The forgetful functor from the category of first-order Boolean doctrines with a small base category to the category of Boolean doctrines with a small base category has a left adjoint.
\end{corollary}

While \cref{c:left-adjoint-small} has a more compact formulation than \cref{t:quantif-complet}, it is, a priori, slightly weaker.
Indeed, while the universal property of the unit in \cref{c:left-adjoint-small} ranges over first-order Boolean doctrines over a \emph{small} category, the universal property of the quantifier completion in \cref{t:quantif-complet} ranges over \emph{all} first-order Boolean doctrines (including more cases, as for instance the subset doctrine).

\begin{remark}
    Requiring a smallness condition for the existence of the quantifier completion is sensible.
    Indeed, the fiber over $X\in\C$ of the hypothetical quantifier completion of a Boolean doctrine $\P \colon \C\op \to \BA$ must be a Boolean algebra (and hence a \emph{set}), and yet it must contain all elements of the form $\fa{Y}{X} \alpha$ with $\alpha \in \P(X \times Y)$ and $Y$ ranging over \emph{all} objects of $\C$. 
    Smallness of $\C$ avoids size issues.
\end{remark}

We recall that a \emph{universal formula} is a formula of the form $\forall x_1\,\dots\,\forall x_n\,\alpha(x_1,\dots,x_n,y_1,\dots,y_m)$ with $\alpha$ quantifier-free, that a \emph{universal sentence} is the universal closure of a quantifier-free formula, i.e.\ a universal formula with no free variables, and that a \emph{universal theory} is a theory consisting of universal sentences.

\begin{proposition}[$\LT^\T$ is the quantifier completion of $\LT_0^\T$, for $\T$ universal]\label{ex:sanitycheck}
    Let $\T$ be a universal theory.
    The pair $(\id_{\ctx},i)\colon\LT^\T_0\to\LT^\T$ (where $i$ is the componentwise inclusion) is a quantifier completion of $\LT^\T_0$ (= the subfunctor of $\LT^\T$ of quantifier-free formulas modulo $\T$, see \cref{ex:motivating-strat-fragment}).
\end{proposition}

\begin{proof}
    The idea is that the identifications of formulas happening in $\LT^\T$ are already encoded in $\LT_0^\T$, since $\T$ is universal.
    
    Let $\R\colon\ctx\op\to\BA$ be a first-order Boolean doctrine and $(\id_{\ctx},\m)\colon\LT^\T_0\to\R$ a Boolean doctrine morphism. We look for a first-order Boolean doctrine morphism $\LT^\T\to\R$.
    Let us consider the family of functions $(\mathbb{I}_n\colon \mathbb{P}_n\to\R(x_1,\dots,x_n), \,R\mapsto \m_{\vec x}([R(\vec x):\vec x]_\T))_n$, and let $\mathcal{I}$ be its extension as in \cref{r:int-fmlas}. It can be easily proved by induction on the complexity of the formulas that for every quantifier-free formula $\alpha:\vec x$ we have $\mathcal{I}(\alpha:\vec x)=\m_{\vec x}([\alpha:\vec x]_\T)$.

    Let $\varphi\in\T$, and let us prove that $\mathcal{I}(\varphi:())=\top_{\R()}$. Since $\T$ is universal, we have $\varphi=\forall\vec y\, \psi$, with $\psi:\vec y$ a quantifier-free formula. By adjunction, from $[\top:()]_\T\leq [\forall\vec y\,\psi:()]_\T$ in $\LT^\T()$ it follows that $[\top:\vec y]_\T\leq [\psi:\vec y]_\T$ in $\LT^\T(\vec y)$. We compute $\mathcal{I}(\varphi:())$:
    \[
    \mathcal{I}(\forall \vec y\,\psi:())= \fa{\vec y}{()}\,\mathcal{I}(\psi:\vec y)=\fa{\vec y}{()}\m_{\vec y}([\psi:\vec y]_\T)=\fa{\vec y}{()}\m_{\vec y}([\top:\vec y]_\T)=\fa{\vec y}{()}(\top_{\R(\vec y)})=\top_{\R()}.
    \]
    By \cref{p:univ-prop-ltlt}, we define $(\id_{\ctx},\n)\colon \LT^\T \to \R$ as the first-order Boolean doctrine morphism associated to the family $(\mathbb{I}_n)_n$, with $\n_{\vec x}([\alpha:\vec x]_\T)=\mathcal{I}(\alpha:\vec x)$ for every $[\alpha:\vec x]_\T\in\LT^\T(\vec x)$. 
    Clearly,
    $(\id_{\ctx},\m) = (\id_{\ctx},\n) \circ (\id_{\ctx}, {i})$.
    
    To prove uniqueness, let $(\id_{\ctx},\n')\colon \LT^\T\to\R$ be a first-order Boolean doctrine morphism such that $(\id_{\ctx},\m) = (\id_{\ctx},\n') \circ (\id_{\ctx}, {i})$. Under the bijection in \cref{p:univ-prop-ltlt}, this morphism corresponds to the family $(\mathbb{I}'_n\colon\mathbb{P}_n\to\R(\vec x),\, R \mapsto \n'_{\vec x}([R(\vec x):\vec x]_\T))_{n}$. Since $R(\vec x):\vec x$ is a quantifier-free formula, we have
    \[
    \n'_{\vec x}([R(\vec x):\vec x]_\T)=\n'_{\vec x}(i[R(\vec x):\vec x]_\T)=\m_{\vec x}([R(\vec x):\vec x]_\T)=\mathcal{I}(R(\vec x):\vec x)=\mathbb{I}_n(R),
    \]
    and thus $(\mathbb{I}'_n)_n=(\mathbb{I}_n)_n$, and hence $(\id_{\ctx},\n')=(\id_{\ctx},\n)$.

    By \cref{t:quantif-complet}, $(\id_{\ctx},i)\colon\LT^\T_0\to\LT^\T$ is a quantifier completion of $\LT^\T_0$.
\end{proof}

\begin{corollary}
    \label{ex:quantif-compl-lt}
    Let $\T$ be a first-order theory.
    The quantifier completion of $\LT^\T_0$ is isomorphic to the syntactic doctrine $\LT^{\mathcal{U}}$, where $\mathcal{U}$ is the theory in the same language of $\T$ whose axioms are all universal sentences derivable from $\T$.
\end{corollary}

\begin{proof}
    
    To prove it, it is enough to show that $\LT^\T_0\cong\LT^{\mathcal{U}}_0$; since $\mathcal{U}$ is a universal theory, it then follows from \cref{ex:sanitycheck} that the quantifier completion of $\LT_0^\T$ is $\LT^{\mathcal{U}}$.

    Each component of the natural transformation $k\colon\LT^{\mathcal{U}}_0\to \LT^\T_0$ maps $[\alpha]_{\mathcal{U}}$ to $[\alpha]_{\T}$. This is the restriction of the natural transformation $\LT^{\mathcal{U}}\to\LT^{\T}$ induced by the inclusion of $\mathcal{U}$ in the deductive closure of $\T$ (see \cref{p:quotient-extension}). The restriction is well-defined because if $\alpha\dashv\vdash_{\mathcal{U}}\alpha'$ with $\alpha'$ quantifier-free, then $\alpha\dashv\vdash_{\T}\alpha'$. Let $\alpha$ be in some fiber of $\LT_0^\T$, i.e. $\alpha\dashv\vdash_{\T}\alpha'$ with $\alpha'$ quantifier-free. Then $k$ sends $[\alpha']_{\mathcal{U}}$ to $[\alpha']_{\T}=[\alpha]_{\T}$, and thus $k$ is componentwise surjective. Now, let $[\alpha:\vec x]_{\mathcal{U}},[\beta:\vec x]_{\mathcal{U}}\in\LT^{\mathcal{U}}_0(\vec x)$, let $\alpha':\vec x$ and $\beta':\vec x$ be quantifier-free formulas such that $\alpha\dashv\vdash_{\mathcal{U}}\alpha'$ and $\beta\dashv\vdash_{\mathcal{U}}\beta'$, and suppose $\alpha\vdash_\T\beta$. Then $\alpha'\vdash_\T\beta'$, so $\vdash_\T\forall\vec x\,(\alpha'\to\beta')$, i.e.\ $\forall\vec x\,(\alpha'\to\beta')\in\mathcal{U}$, so $\alpha'\vdash_{\mathcal{U}}\beta'$, hence $\alpha\vdash_{\mathcal{U}}\beta$, thus $k$ componentwise reflects the order, and so it is componentwise injective.
\end{proof}

\begin{remark}[Comparison with the existential completion of a primary doctrine]\label{r:trotta}
    We note here that the quantifier completion of a Boolean doctrine $\P$ may differ from the existential completion (in the sense of \cite{Trotta2020}) of $\P$ seen as a primary doctrine.
    We give more details for the interested reader.

    A \emph{primary doctrine} is a functor $\P \colon \C\op \to \mathsf{InfSL}$ where $\C$ is a category with finite products and $\mathsf{InfSL}$ is the category of inf-semilattices.
    An \emph{existential doctrine} is a primary doctrine $\P \colon \C\op \to \mathsf{InfSL}$ such that the monotone function $\P(\pr_1)\colon \P(X)\to\P(X\times Y)$ has a left adjoint $\ex{Y}{X}$ satisfying the Beck-Chevalley condition and the Frobenius reciprocity.
    The existential completion of a primary doctrine $\P$ consists of an existential (and thus primary) doctrine $\P^e$ and a primary doctrine morphism $(\id_\C,\iota)\colon \P \to \P^e$ satisfying a suitable universal property.
    Roughly speaking, $\P^e$ freely adds the existential quantifier to $\P$.

    The difference between the existential completion of a primary doctrine and the quantifier completion of a Boolean doctrine is that in the former case there is no alternation of quantifiers, and so the existential completion can be done in one step with a simple description: formulas in the existential completion of $\P$ can be described very explicitly as $\exists \vec x \, \alpha(\vec x,\vec y)$, with $\alpha \in \P(\vec{x},\vec{y})$.
    
    Such a simple description of the quantifier completion is not directly available in our setting, because formulas can have arbitrarily large quantifier alternation depth.
    One way to still get a simple description might be obtained by adding layers of quantifier alternation depth one at a time (see \cite{AbbadiniGuffanti2} for some partial results in this direction).
    Indeed, part of the motivation for our study of QA-stratified Boolean doctrines is our desire to obtain a step-by-step description of the quantifier completion.
\end{remark}

\begin{theorem} \label{p:qff-of-completion}
    Every Boolean doctrine over a small category is a quantifier-free fragment of its quantifier completion.
\end{theorem}
\begin{proof}
    Let $\P$ be a Boolean doctrine over a small category $\C$.
    By \cref{t:quantif-complet}, we know that $\P$ has a quantifier completion, namely $(\id_\C,\eta_\P)\colon \P\to UF(\P)$, where $\eta$ is the componentwise injective unit of the adjunction $F\dashv U$ in \cref{t:left-adjoint}.
    It is a standard fact that the image of each component of the unit of an adjunction having as right adjoint a forgetful functor between varieties of algebras is generating.
    Therefore, the image of $\eta_\P$ generates $UF(\P)$.
    Hence, $\P$ is a quantifier-free fragment of $UF(\P)$.
\end{proof}

\section{The case with equality} \label{s:equality}

What happens if the language has equality?
In this section, we adapt the main definitions and results of the previous sections to the setting with equality.
Concerning definitions, essentially we just need to require quantifier-free fragments to contain equalities of terms (\cref{d:whole-and-first-layer=}).
Concerning results, a pleasant realization is
that the situation does not change much: in particular, a quantifier completion of a Boolean doctrine with equality has equality (\cref{c:equality-only-p0}).

We start by recalling a modern definition of elementarity in the context of Boolean doctrines\footnote{Although these definitions are usually stated for primary doctrines \cite{MaiettiRosolini13,MaiettiRosolini15}, we give them in the Boolean case.}; this notion, which originates in Lawvere's work \cite{Lawvere70}, amounts to the possibility of interpreting equality.

\begin{definition}[\cite{MaiettiRosolini13}]\label{def:equality-classical}
    A Boolean doctrine $\P\colon\C\op\to\BA$ is \emph{elementary} if for every $X\in \C$ there is an element $\delta_X\in\P(X\times X)$ such that, for every $Y\in\C$, the assignment
    \begin{align*}
       \text{\AE}^X_Y\colon \P(Y \times X)&\longrightarrow \P(Y\times X\times X)\\
       \alpha &\longmapsto\P(\ple{\pr_1,\pr_2})(\alpha)\land \P(\ple{\pr_2,\pr_3})(\delta_X)
    \end{align*}
        is a left adjoint of the function 
    \[  \P(\id_Y\times\Delta_X) \colon \P(Y\times X\times X) \to \P(Y \times X).\]
\end{definition}

Equivalently, as shown in \cite[Prop.~2.5]{EmPaRo20},

\begin{definition}\label{def:equality}
    A Boolean doctrine $\P\colon\C\op\to\BA$ is \emph{elementary} if there is a family of elements $(\delta_X)_{X\in\C}$, with $\delta_X\in\P(X\times X)$ for each $X\in \C$, such that for every $X,Y\in\C$
    \begin{enumerate}
        \item\label{i:def=1} $\top_{\P(X)}\leq\P(\Delta_X)(\delta_X)$ in $\P(X)$;
        \item\label{i:def=2} for every $\alpha\in\P(X)$, $\P(\pr_1)(\alpha)\land \delta_X\leq \P(\pr_2)(\alpha)$ in $\P(X\times X)$;
        \item\label{i:def=3} $\P(\ple{\pr_1,\pr_3})(\delta_X)\land \P(\ple{\pr_2,\pr_4})(\delta_Y)\leq \delta_{X\times Y}$ in $\P(X\times Y\times X\times Y)$.
    \end{enumerate}
\end{definition}

The proof of \cite[Prop.~2.5]{EmPaRo20} shows more: a family $(\delta_X)_{X\in\C}$ satisfies the condition in \cref{def:equality-classical} if and only if it satisfies the conditions in \cref{def:equality}.

Informally, condition \eqref{i:def=1} in \cref{def:equality} is the reflexivity of the equality relation, i.e.\ $\vdash x=x$. Condition \eqref{i:def=2} is the substitutivity property, i.e.\ $\alpha(x)\land (x=x')\vdash \alpha(x')$. Finally, condition \eqref{i:def=3} can be roughly interpreted as $``(x=x')\land (y=y')\vdash (x,y)=(x',y')"$, meaning that two pairs coincide if both entries do.

\begin{remark}
    There is a unique family $(\delta_X)_{X\in\C}$ satisfying the conditions in \cref{def:equality-classical} (or, equivalently, in \cref{def:equality}).
    In fact, for every Boolean doctrine $\P$ and every $X \in \C$, there is at most one element $\delta_X \in \P(X \times X)$ satisfying the condition in \cref{def:equality-classical}.
    Indeed, for one such $\delta_X$, since $\delta_X=\text{\AE}^X_\tmn(\top_{\P(X)})$, we can observe that for every $\beta\in\P(X\times X)$ we have
    \begin{equation} \label{eq:describe}
    \delta_X\leq\beta\iff \top_{\P(X)}\leq\P(\Delta_X)(\beta).
    \end{equation}
    This condition determines the element $\delta_X$ uniquely, as \eqref{eq:describe} describes its principal upset.
\end{remark}

For every $X\in\C$, the element $\delta_X$ is called \emph{fibered equality on $X$}.

\begin{remark}\label{r:delta-symm}
    For every elementary Boolean doctrine $\P$ and every object $X$ in its base category we have
    \[
    \delta_X \leq \P(\ple{\pr_2,\pr_1})(\delta_X),
    \]
    which states the symmetry of the equality relation, i.e.\ $x=x'\vdash x'=x$.
    This is immediate from the equality $\delta_X=\text{\AE}^X_\tmn(\top_{\P(X)})$, the adjunction $\text{\AE}^X_\tmn\dashv\P(\Delta_X)$ and the equality $\ple{\pr_2,\pr_1}\circ\Delta_X=\Delta_X$.
\end{remark}

\begin{definition}
    An \emph{elementary Boolean doctrine morphism} from $\P\colon\C\op\to \BA$ to $\R \colon \cat{D}\op\to \BA$ is a Boolean doctrine morphism $(M,\m)\colon \P\to\R$ that preserves fibered equalities, i.e.\ such that, for every $X\in\C$, $\m_{X\times X}(\delta^\P_X)=\delta^\R_{M(X)}$.
\end{definition}

\begin{definition}
    An \emph{elementary first-order Boolean doctrine} $\P\colon\C\op\to\BA$ is a first-order Boolean doctrine $\P$ that is elementary.
    
    An \emph{elementary first-order Boolean doctrine morphism} from $\P$ to $\R$ is a first-order Boolean doctrine morphism $(M,\m)\colon \P\to\R$ that is elementary.
\end{definition}

\begin{example}\label{fbf=}
    For a first-order theory $\T$ in a language with equality one can define a syntactic doctrine $\LT_{=}^\T$ similarly to \cref{fbf}. In this case, one has to consider also the formulas involving equality, and take the equivalence relation of equiprovability modulo $\T$ in first-order logic with equality.
    The syntactic doctrine $\LT_{=}^\T$ is an elementary first-order Boolean doctrine: for every context $\vec x = (x_1,\dots,x_n)$, the fibered equality $\delta_{\vec x}\in \LT_=^\T(x_1,\dots,x_n,x'_1,\dots,x'_n)$ is the formula $(x_1=x'_1)\land\dots\land (x_n=x'_n)$.
\end{example}

\begin{example}
    The subset doctrine $\mathscr{P}\colon\Set\op\to\BA$ is elementary: for every set $X$, the fibered equality $\delta_X\in \mathscr{P}(X\times X)$ is the subset $\Delta_X=\{(x,x)\mid x \in X\}\subseteq X\times X$.
\end{example}

The following captures the idea that the elementarity of a first-order Boolean doctrine can be checked in a quantifier-free fragment; this boils down to the fact that the principle of substitutivity \eqref{i:def=2} in \cref{def:equality}---i.e.\ $\alpha(x) \land (x = x') \to \alpha(x')$---can be checked just for those $\alpha$ that are quantifier-free.

\begin{proposition}\label{p:equality-only-p0}
    Let $\P_0$ be a quantifier-free fragment of a first-order Boolean doctrine $\P\colon\C\op\to\BA$. The following are equivalent.
    \begin{enumerate}
        \item \label{i:elementary-everywhere}
        The first-order Boolean doctrine $\P$ is elementary, and for every $X\in\C$ the fibered equality $\delta_X$ belongs to $\P_0(X\times X)$.
        \item \label{i:elementary-on-layer-0}
        The Boolean doctrine $\P_0$ is elementary.
    \end{enumerate}
    Moreover, under these conditions, the fibered equalities of $\P$ and $\P_0$ coincide. In other words, the Boolean doctrine morphism defined by the componentwise inclusion $(\id_\C,i)\colon\P_0\hookrightarrow\P$ is elementary.
\end{proposition}
\begin{proof}
    \eqref{i:elementary-everywhere} $\Rightarrow$ \eqref{i:elementary-on-layer-0}. It is easy to see that conditions (\ref{i:def=1}--\ref{i:def=3}) in \cref{def:equality} hold in $\P_0$ if they hold in $\P$ (with respect to the same fibered equalities).
    
    \eqref{i:elementary-on-layer-0} $\Rightarrow$ \eqref{i:elementary-everywhere}. Suppose that there is a family $(\delta_X)_{X\in\C}$ satisfying the conditions (\ref{i:def=1}--\ref{i:def=3}) in \cref{def:equality} with $\P_0$ instead of $\P$. We show that the same conditions hold for $\P$, as well (with respect to the same fibered equalities). The conditions \eqref{i:def=1} and \eqref{i:def=3} follow from the fact that $\P_0$ is a Boolean subfunctor of $\P$. We prove the condition \eqref{i:def=2} by induction on the complexity of $\alpha\in\P(X)$, recalling that $\P_0$ generates $\P$ by Boolean combinations and universal quantifications.
    \begin{itemize}
        \item If $\alpha\in\P_0(X)\subseteq\P(X)$ is a generator, then the condition \eqref{i:def=2} holds by assumption.
        (Note that, in particular, \eqref{i:def=2} holds for $\alpha=\top_{\P(X)}$.)
        \item If $\alpha=\beta\land\gamma\in\P(X)$ for some $\beta,\gamma\in\P(X)$ such that in $\P(X\times X)$
        \[
        \P(\pr_1)(\beta)\land \delta_X\leq \P(\pr_2)(\beta)\quad\text{and}\quad \P(\pr_1)(\gamma)\land \delta_X\leq \P(\pr_2)(\gamma),
        \]
        then
        \[
        \P(\pr_1)(\alpha)\land \delta_X =\P(\pr_1)(\beta)\land\P(\pr_1)(\gamma)\land \delta_X \leq \P(\pr_2)(\beta)\land \P(\pr_2)(\gamma)= \P(\pr_2)(\alpha).
        \]
        \item If $\alpha=\lnot\beta\in\P(X)$ for some $\beta\in\P(X)$ such that in $\P(X\times X)$
        \[
        \P(\pr_1)(\beta)\land \delta_X\leq \P(\pr_2)(\beta),
        \]
        then, applying the reindexing along $\ple{\pr_2,\pr_1}\colon X\times X\to X\times X$ to both sides, we obtain
        \[
        \P(\pr_2)(\beta)\land \P(\ple{\pr_2,\pr_1})(\delta_X)\leq \P(\pr_1)(\beta),
        \]
        from which we get
        \[
        \P(\pr_1)(\lnot\beta)\land \P(\ple{\pr_2,\pr_1})(\delta_X)\leq \P(\pr_2)(\lnot\beta).
        \]
        Since $\delta_X\leq\P(\ple{\pr_2,\pr_1})(\delta_X)$ (\cref{r:delta-symm}) we obtain
        \[
            \P(\pr_1)(\alpha)\land\delta_X\leq\P(\pr_1)(\lnot\beta)\land\P(\ple{\pr_2,\pr_1})(\delta_X)\leq  \P(\pr_2)(\lnot\beta)= \P(\pr_2)(\alpha).
        \]
        
        \item If $\alpha=\fa{Y}{X}\beta\in\P(X)$ for some $\beta\in\P(X\times Y)$ such that in $\P(X\times Y\times X\times Y)$
        \begin{equation}\label{eq:elem1}
        \P(\ple{\pr_1,\pr_2})(\beta)\land\delta_{X\times Y}\leq\P(\ple{\pr_3,\pr_4})(\beta),
        \end{equation}
        then we want to prove that in $\P(X)$
        \begin{equation} \label{eq:elem2}
          \P(\pr_1)(\fa{Y}{X}\beta)\land \delta_X\leq \P(\pr_2)(\fa{Y}{X}\beta).
        \end{equation}
        By the Beck-Chevalley condition, we rewrite \eqref{eq:elem2} as
        \begin{equation*}
          \P(\pr_1)(\fa{Y}{X}\beta)\land \delta_X\leq \fa{Y}{X\times X}(\P(\ple{\pr_2,\pr_3})(\beta)),
        \end{equation*}
        and so by the adjunction $\P(\ple{\pr_1,\pr_2})\dashv\fa{Y}{X\times X}$ we are left to prove in $\P(X\times X\times Y)$
        \begin{equation}\label{eq:elem6}
          \P(\pr_1)(\fa{Y}{X}\beta)\land \P(\ple{\pr_1,\pr_2})(\delta_X)\leq \P(\ple{\pr_2,\pr_3})(\beta).
        \end{equation}
        Observe that in $\P(X\times Y\times X\times Y)$, using the condition \eqref{i:def=3} in \cref{def:equality} and \eqref{eq:elem1} we have
        \begin{equation}\label{eq:elem3}
        \P(\ple{\pr_1,\pr_2})(\beta)\land\P(\ple{\pr_1,\pr_3})(\delta_X)\land \P(\ple{\pr_2,\pr_4})(\delta_Y)\leq\P(\ple{\pr_3,\pr_4})(\beta).
        \end{equation}
        Reindexing both sides of \eqref{eq:elem3} along $\ple{\pr_1,\pr_3,\pr_2,\pr_3}\colon X\times X\times Y\to X\times Y\times X\times Y$ we obtain in $\P( X\times X\times Y)$
        \begin{equation*}
        \P(\ple{\pr_1,\pr_3})(\beta)\land\P(\ple{\pr_1,\pr_2})(\delta_X)\land \P(\ple{\pr_3,\pr_3})(\delta_Y)\leq\P(\ple{\pr_2,\pr_3})(\beta),
        \end{equation*}
        which we reduce, using the condition \eqref{i:def=1} in \cref{def:equality}, to
        \begin{equation}\label{eq:elem4}
        \P(\ple{\pr_1,\pr_3})(\beta)\land\P(\ple{\pr_1,\pr_2})(\delta_X)\leq\P(\ple{\pr_2,\pr_3})(\beta).
        \end{equation}
        Observe that by the adjunction $\P(\pr_1)\dashv\fa{Y}{X}$ we have in $\P(X\times Y)$ 
        \begin{equation}\label{eq:elem5}
            \P(\pr_1)(\fa{Y}{X}\beta)\leq\beta.
        \end{equation}
        Reindexing both sides of \eqref{eq:elem5} along $\ple{\pr_1,\pr_3}\colon X\times X\times Y\to X\times Y$, we obtain in $\P( X\times X\times Y)$
        \begin{equation}\label{eq:elem7}
        \P(\pr_1)(\fa{Y}{X}\beta)\leq\P(\ple{\pr_1,\pr_3})(\beta).
        \end{equation}
        We now have the ingredients to show \eqref{eq:elem6}:
        \begin{align*}
          \P(\pr_1)(\fa{Y}{X}\beta)\land \P(\ple{\pr_1,\pr_2})(\delta_X)&\leq \P(\ple{\pr_1,\pr_3})(\beta)\land \P(\ple{\pr_1,\pr_2})(\delta_X)&&\text{(by \eqref{eq:elem7})}\\
          &\leq \P(\ple{\pr_2,\pr_3})(\beta) &&\text{(by \eqref{eq:elem4})}.
        \end{align*}
    \end{itemize}

    The above proofs of the two implications also prove the last part of the statement.
\end{proof}

We enrich \cref{d:whole-and-first-layer,d:only-layers} with the elementary structure, imposing that every fibered equality is a quantifier-free formula.

\begin{definition}[Elementary quantifier-free fragment]\label{d:whole-and-first-layer=}\hfill
    \begin{enumerate}
        \item 
        Let $\P\colon\C\op\to\BA$ be an elementary first-order Boolean doctrine. An \emph{elementary quantifier-free fragment of $\P$} is a quantifier-free fragment $\P_0$ of $\P$ such that, for every $X\in\C$, $\delta_X\in\P_0(X\times X)$.
        
        \item Let $\P$ and $\R$ be elementary first-order Boolean doctrines, and let $\P_0$ and $\R_0$ be elementary quantifier-free fragments of $\P$ and $\R$, respectively.
        A \emph{morphism of elementary quantifier-free fragments} from $(\P_0, \P)$ to $(\R_0, \R)$ is a morphism $(M, \m) \colon (\P_0,\P )\to (\R_0,\R)$ of quantifier-free fragments such that $(M, \m) \colon \P\to \R$ is an elementary first-order Boolean doctrine morphism.
        
        \item We let $\QFF_=$ denote the subcategory of $\QFF$ whose objects are pairs $(\P_0, \P)$ with $\P_0$ an elementary quantifier-free fragment of an elementary first-order Boolean doctrine $\P$ and whose morphisms are morphisms of elementary quantifier-free fragments.
    \end{enumerate}
\end{definition}

\begin{example}\label{ex:motivating-strat-fragment=}
    Similarly to \cref{ex:motivating-strat-fragment}, for a theory $\T$ in a language with equality, the functor $\LT_{=,0}^\T$ of quantifier-free formulas modulo $\T$ is an elementary quantifier-free fragment of the syntactic doctrine $\LT_=^\T$.
\end{example}

\begin{proposition}
    Let $i\colon\P_0\hookrightarrow\P_1$ be a QA-one-step Boolean doctrine. The following are equivalent.
    \begin{enumerate}
        \item
        The Boolean doctrine $\P_1$ is elementary, and, for each $X\in\C$, the fibered equality $\delta_X$ belongs to $\P_0(X\times X)$.
        \item
        The Boolean doctrine $\P_0$ is elementary.
    \end{enumerate}
    Moreover, under these conditions, the fibered equalities of $\P_1$ and $\P_0$ coincide. In other words, the Boolean doctrine morphism defined by the componentwise inclusion $(\id_\C,i)\colon\P_0\hookrightarrow\P_1$ is elementary.
\end{proposition}
\begin{proof}
    The proof is similar to the proof of \cref{p:equality-only-p0}.
\end{proof}

\begin{definition}[Elementary QA-one-step Boolean doctrine] \label{d:one-step-doctrine-eq}
    An \emph{elementary QA-one-step Boolean doctrine} is a QA-one-step Boolean doctrine $i\colon\P_0 \hookrightarrow \P_1$ with $\P_0$ elementary.
\end{definition}

\begin{definition}[Elementary QA-stratified Boolean doctrine] \hfill
    \begin{enumerate}
        \item An \emph{elementary QA-stratified Boolean doctrine} is a QA-stratified Boolean doctrine
        \[
        \P_0\hookrightarrow \P_1 \hookrightarrow \P_2 \hookrightarrow \dots
        \]
        with $\P_0$ elementary.
    
        \item A \emph{morphism of elementary QA-stratified Boolean doctrines} from $(\P_n)_{n\in\N}$ to $(\R_n)_{n\in\N}$ is
        a morphism 
        \[
        (M, (\m_n)_{n\in\N}) \colon (\P_n)_{n\in\N}\to (\R_n)_{n\in\N}
        \]
        of QA-stratified Boolean doctrines such that $\m_0$ preserves fibered equalities.
        
        \item We let $\QA_{=}$ denote the subcategory of $\QA$ whose objects are elementary QA-stratified Boolean doctrines and whose morphisms are morphisms of elementary QA-stratified Boolean doctrines.
    \end{enumerate}
\end{definition}

\begin{proposition}
    The equivalence between $\QFF$ and $\QA$ (\cref{p:equivalent-cat}) restricts to an equivalence between $\QFF_=$ and $\QA_{=}$.
\end{proposition}
\begin{proof}
    The fact that the restrictions $U'\colon \QFF_{=}\to\QA_=$ and $L'\colon \QA_=\to \QFF_{=}$ of the functors $U\colon \QFF\to\QA$ and $L\colon \QA\to \QFF$ defined in the proof of \cref{p:equivalent-cat} are well-defined follows from \cref{p:equality-only-p0}.
\end{proof}

Finally, we adapt the results about the quantifier completion (\cref{sec:quantifier compl}) to the case with equality.

\begin{theorem}\label{c:equality-only-p0}
    Let $\P\colon\C\op\to\BA$ be a Boolean doctrine and $(\id_\C,\mathfrak{i})\colon \P\to\P^\EA$ a quantifier completion of $\P$. If $\P$ is elementary, then $\P^\EA$ is elementary, and the Boolean doctrine morphism $(\id_\C,\mathfrak{i})$ is elementary.
\end{theorem}
\begin{proof}
    This follows from \cref{p:equality-only-p0,p:qff-of-completion}.
\end{proof}

\Cref{c:equality-only-p0} is similar to the fact that the existential completion of an elementary primary doctrine is elementary \cite[Prop.~6.1]{Trotta2020}.

\begin{remark}
    Another consequence of \cref{p:equality-only-p0} is that the quantifier completion of an elementary doctrine satisfies the version of the universal property of the quantifier completion (\cref{d:quantifier-completion}) in which every doctrine and every morphism are required to be elementary.
\end{remark}

\begin{corollary}
    The forgetful functor from the category of elementary first-order Boolean doctrines with a small base category to the category of elementary Boolean doctrines
    with a small base category has a left adjoint.
\end{corollary}

\begin{example}\label{ex:sanitycheck=}
    Let $\T$ be a universal theory in a language with equality.
    Let $\LT_=^\T\colon\ctx\op\to\BA$ be the syntactic doctrine of $\T$ (as in \cref{fbf=}) and $\LT^\T_{=,0}$ be the elementary quantifier-free fragment of $\LT_=^\T$ of quantifier-free formulas modulo $\T$ (as in \cref{ex:motivating-strat-fragment=}).
    Then, the morphism $(\id_{\ctx},i)\colon\LT^\T_{=,0}\hookrightarrow\LT_=^\T$ (where $i\colon\LT^\T_{=,0}\hookrightarrow\LT_=^\T$ is the componentwise inclusion) is a quantifier completion of $\LT^\T_{=,0}$.
    To prove this, the idea is that one can piggyback on the case without equality (\cref{ex:sanitycheck}) by considering equality as an ``ordinary'' binary relation symbol satisfying the classical axioms of equality:
    \begin{enumerate}
        \item\label{i:ex-refl=}
        $\forall \vec x \, (t(\vec x )=t(\vec x))$ for every term $t(\vec x)$;
        \item\label{i:ex-alpha=}
        $\forall \vec x \, ((t(\vec x)=u(\vec x))\land \alpha(\vec x,t(\vec x)/y)\to\alpha(\vec x, u(\vec x)/y))$ for every formula $\alpha(\vec x,y)$ and for all terms $t(\vec x)$ and $u(\vec x)$ substitutable for $y$.
    \end{enumerate}
    It can be shown by induction on the complexity of the formulas that the scheme \eqref{i:ex-alpha=} can be replaced with
    \begin{enumerate}[label = (\arabic*$'$), ref = \arabic*$'$, start = 2]
        \item\label{i:ex-alpha0=} $\forall \vec x \, ((t(\vec x)=u(\vec x))\land \alpha(\vec x,t(\vec x)/y)\to\alpha(\vec x,u(\vec x)/y))$ for every \emph{quantifier-free} $\alpha(\vec x,y)$ and for all terms $t(\vec x)$ and $u(\vec x)$ substitutable for $y$.
    \end{enumerate}
    Since the axioms in \eqref{i:ex-refl=} and \eqref{i:ex-alpha0=} are universal sentences, we can apply the result in \cref{ex:sanitycheck}.
\end{example}

\begin{example}
    If $\T$ is a first-order theory in a language with equality, 
    a quantifier completion of $\LT^\T_{=,0}$
    is the syntactic doctrine $\LT_=^{\mathcal{U}}$, where $\mathcal{U}$ is the theory in the same language $\L$ whose axioms are all universal sentences derivable from $\T$.
    To prove this, the idea is that one can piggyback again on the case without equality (\cref{ex:quantif-compl-lt}), using a similar argument to the one in \cref{ex:sanitycheck=}.
\end{example}

\begin{theorem}
     Every elementary Boolean doctrine over a small category is an elementary quantifier-free fragment of an elementary first-order Boolean doctrine (namely, its quantifier completion).
\end{theorem}

\begin{proof}
    Let $\P$ be an elementary Boolean doctrine over a small category. By \cref{p:qff-of-completion}, $\P$ is a quantifier-free fragment of its quantifier completion $\P^\EA$. By \cref{p:equality-only-p0}, $\P^\EA$ is elementary, with the same fibered equalities as $\P$. Thus, $\P$ is an elementary quantifier-free fragment of $\P^\EA$. 
\end{proof}

This shows that, under a smallness assumption, the structures occurring as elementary quantifier-free fragments are precisely the elementary Boolean doctrines.

\section{Future work}\label{s:future-work}

In this paper, we presented an algebraic/categorical approach to quantifier-free formulas and quantifier alternation depth modulo a theory.
In \cref{sec:char-qff}, we characterized the structures occurring as quantifier-free fragments as precisely the Boolean doctrines (\cref{t:character-0th-layer}).
This is the first step of a larger program of finding the intrinsic algebraic/categorical structure of each level of the quantifier alternation hierarchy:
\begin{question}
    For each $m \in \N$, what are the structures $(\P_0, \dots, \P_m)$ that are part of some QA-stratification of a first-order Boolean doctrine, or, equivalently, of a QA-stratified Boolean doctrine $(\P_n)_{n \in \N}$?
\end{question}

While in this paper we settled the case $m = 0$ (Boolean doctrines), the remaining ones are yet to be answered.
To answer the question in its generality, it is enough to solve the case $m = 1$, since only the last level matters.
Our conjecture is that the structures $(\P_0, \P_1)$ that are part of a QA-stratified Boolean doctrine are precisely the QA-one-step Boolean doctrines.
To verify the correctness of this conjecture, a completeness theorem similar to G\"odel’s completeness theorem might be useful. (G\"odel’s theorem would be recovered as the case $\P_0 = \P_1$.)
A further interesting problem is to give a constructive proof of the fact that every QA-one-step Boolean doctrine embeds into some first-order Boolean doctrine; an option might be using Boolean-valued models and canonical extensions along the lines of \cref{t:character-0th-layer}.

Furthermore, a long-term goal is to provide an explicit layer-by-layer construction of the quantifier completion.
In the preprint \cite{AbbadiniGuffanti2}, we show how to obtain the first step, i.e.\ how to obtain the layer $\P_1$ free over $\P_0$.
This amounts to a doctrinal version of Herbrand's theorem restricted to formulas of quantifier-alternation depth at most  $1$.
An interesting direction to pursue is whether one can give a version of Herbrand’s theorem to build the free ``next level'' $\P_{n+2}$ from $\P_n \subseteq \P_{n+1}$. 
(This may lead to another constructive proof that QA-one-step Boolean doctrines embed in first-order Boolean doctrines.)

Moreover, motivated by \cite{GehrkeJaklEtAL2023}, we plan to obtain a Stone dual description (in terms of Joyal's polyadic spaces \cite{Joyal1971,Marques2023,vanGoolMarques2024}) of the step-by-step construction of the quantifier completion, similarly to what has been done for adding layers of modality to Boolean algebras.

Finally, in a different direction, we report here the question from \cref{r:comparison}: is the notion of quantifier-free formula modulo a universal theory intrinsic? In other words,
\begin{question} \label{q:intrinsic}
    Does any isomorphism between the quantifier-completions of two Boolean doctrines $\P$ and $\R$ restrict to an isomorphism between $\P$ and $\R$?
\end{question}

\appendix

\section{Classical first-order sequent calculus (with contexts)}\label{app:calculus}

We recall here the sequent calculus \emph{with contexts} for classical first-order logic; here, each sequent is equipped with a finite list $\vec y = (y_1, \dots, y_m)$ of distinct variables\footnote{One may use finite sets in place of finite lists of distinct variables.}, and the free variables of every formula in the sequent belong to $\{y_1, \dots, y_m\}$.
Compared to the usual calculus, the key difference is that some of the rules for quantifiers change the context.
Additionally, the empty set is a possible model in this calculus: for example, the sequent $\Rightarrow \exists x \top$ in the empty context is not provable.
This calculus and its variations circulate among researchers in categorical logic since the '70s, and are embodied more or less explicitly in the texts treating the internal language of toposes and other categories of logic.
Some references are 
\cite[Part~II, Sec.~1]{Lambek1988} (which treats the intuitionistic case),
\cite{sequentcalculus} and \cite{BertaMeloni1987}.

We write $\alpha : \vec y$ to mean that the free variables of the first-order formula $\alpha$ belong to $\{y_1, \dots, y_m\}$. For a set $\Gamma$ of formulas, we write $\Gamma : \vec{y}$ to mean that, for each $\alpha \in \Gamma$, we have $\alpha : \vec{y}$.
In the following, $\alpha,\beta, \alpha_0,\alpha_1 : \vec y$ and $\zeta:(\vec y,x)$ (where $x \notin \{y_1, \dots, y_n\}$) are formulas, and $\Gamma, \Delta,\Gamma_1,\Gamma_2,\Delta_1,\Delta_2:\vec y$ finite lists of formulas.

\smallskip

\begin{center}
   \fbox{Structural Rules}
\end{center}
\begin{minipage}{0.5\textwidth}
    \begin{prooftree}
        \AxiomC{$\Gamma\Rightarrow_{\vec y}\Delta$}
        \RightLabel{\ (LW)}
        \UnaryInfC{$\alpha,\,\Gamma\Rightarrow_{\vec y}\Delta$}
    \end{prooftree}
\end{minipage}%
\begin{minipage}{0.5\textwidth}
    \begin{prooftree}
        \AxiomC{$\Gamma\Rightarrow_{\vec y}\Delta$}
        \RightLabel{\ (RW)}
        \UnaryInfC{$\Gamma \Rightarrow_{\vec y} \Delta,\, \alpha$}
    \end{prooftree}
\end{minipage}%
    
\noindent
\begin{minipage}{0.5\textwidth}
    \begin{prooftree}
        \AxiomC{$\alpha, \,\alpha, \,\Gamma \Rightarrow_{\vec y} \Delta$}
        \RightLabel{\ (LC)}
        \UnaryInfC{$\alpha, \,\Gamma \Rightarrow_{\vec y} \Delta$}
    \end{prooftree}
\end{minipage}%
\begin{minipage}{0.5\textwidth}
    \begin{prooftree}
        \AxiomC{$\Gamma \Rightarrow_{\vec y} \Delta,\, \alpha,\, \alpha$}
        \RightLabel{\ (RC)}
        \UnaryInfC{$\Gamma \Rightarrow_{\vec y}\Delta,\, \alpha$}
    \end{prooftree}
\end{minipage}%

\noindent
\begin{minipage}{0.5\textwidth}
    \begin{prooftree}
        \AxiomC{$\alpha, \,\beta, \,\Gamma \Rightarrow_{\vec y} \Delta$}
        \RightLabel{\ (LE)}
        \UnaryInfC{$\beta,\,\alpha, \,\Gamma \Rightarrow_{\vec y} \Delta$}
    \end{prooftree}
\end{minipage}%
\begin{minipage}{0.5\textwidth}
    \begin{prooftree}
        \AxiomC{$\Gamma \Rightarrow_{\vec y} \Delta,\, \alpha,\, \beta$}
        \RightLabel{\ (RE)}
        \UnaryInfC{$\Gamma \Rightarrow_{\vec y}\Delta,\,\beta,\, \alpha$}
    \end{prooftree}
\end{minipage}%

\bigskip

\noindent
\begin{minipage}[t]{0.4\textwidth}
    \begin{center}
    \fbox{Cut Rule}
\end{center}
\begin{center}
    \begin{prooftree}
        \AxiomC{$\Gamma_1\Rightarrow_{\vec y}\Delta_1,\, \alpha$}\AxiomC{$\alpha,\,\Gamma_2\Rightarrow_{\vec y}\Delta_2$}
        \RightLabel{\ (Cut)}
        \BinaryInfC{$\Gamma_1, \,\Gamma_2 \Rightarrow_{\vec y} \Delta_1,\, \Delta_2$}
    \end{prooftree}
\end{center}
\end{minipage}%
\begin{minipage}[t]{0.2\textwidth}
    \begin{center}
    \fbox{Identity Axioms}
\end{center}
\begin{prooftree}
    \AxiomC{}
    \UnaryInfC{$\alpha\Rightarrow_{\vec y}\alpha$}
\end{prooftree}
\end{minipage}%
\begin{minipage}[t]{0.4\textwidth}
    \begin{center}
    \fbox{Context Enlargement Rule}
\end{center}
\begin{center}
    \begin{prooftree}
        \AxiomC{$\Gamma \Rightarrow_{\vec y} \Delta$}
        \UnaryInfC{$\Gamma \Rightarrow_{(\vec y, x)}\Delta$}
    \end{prooftree}
    \begin{center}
        (we stress that $\Gamma:\vec y$, $\Delta:\vec y$, $x\notin\{y_1,\dots,y_m\}$)
    \end{center}
\end{center}
\end{minipage}%

\smallskip
    
\begin{center}
    \fbox{Logical Rules for Connectives}
\end{center}
\begin{minipage}{0.5\textwidth}
    \phantom{a}
\end{minipage}%
\begin{minipage}{0.5\textwidth}
    \begin{prooftree}
        \AxiomC{}
        \RightLabel{\ (R$\top$)}\UnaryInfC{$\Gamma\Rightarrow_{\vec y}\Delta,\,\top$}
    \end{prooftree}
\end{minipage}%

\noindent
\begin{minipage}{0.5\textwidth}
     \begin{prooftree}
        \AxiomC{}
        \RightLabel{\ (L$\bot$)}
        \UnaryInfC{$\bot, \,\Gamma\Rightarrow_{\vec y}\Delta$}
    \end{prooftree}
\end{minipage}
\begin{minipage}{0.5\textwidth}
\end{minipage}%

\noindent
\begin{minipage}{0.5\textwidth}
    \begin{prooftree}
        \AxiomC{$\alpha_i,\,\Gamma\Rightarrow_{\vec y}\Delta$}
        \RightLabel{\ $i=0,1$ (L$\land$)}
        \UnaryInfC{$\alpha_0 \land \alpha_1,\,\Gamma\Rightarrow_{\vec y}\Delta$}
    \end{prooftree}
\end{minipage}%
\begin{minipage}{0.5\textwidth}
    \begin{prooftree}
        \AxiomC{$\Gamma\Rightarrow_{\vec y}\Delta, \,\alpha$}\AxiomC{$\Gamma\Rightarrow_{\vec y}\Delta, \,\beta$}
        \RightLabel{\ (R$\land$)}
        \BinaryInfC{$\Gamma\Rightarrow_{\vec y}\Delta, \,\alpha\land\beta$}
    \end{prooftree}
\end{minipage}%
    
\noindent
\begin{minipage}{0.5\textwidth}
    \begin{prooftree}
        \AxiomC{$\alpha, \,\Gamma\Rightarrow_{\vec y}\Delta, $}\AxiomC{$\beta,\,\Gamma\Rightarrow_{\vec y}\Delta$}
        \RightLabel{\ (L$\lor$)}
        \BinaryInfC{$\alpha \lor \beta, \, \Gamma \Rightarrow_{\vec y} \Delta$}
    \end{prooftree}
\end{minipage}%
\begin{minipage}{0.5\textwidth}
    \begin{prooftree}
        \AxiomC{$\Gamma\Rightarrow_{\vec y}\Delta, \,\alpha_i$}
        \RightLabel{\ $i=0,1$ (R$\lor$)}
        \UnaryInfC{$\Gamma\Rightarrow_{\vec y}\Delta, \,\alpha_0\lor\alpha_1,\,$}
    \end{prooftree}
\end{minipage}%

\noindent
\begin{minipage}{0.5\textwidth}
    \begin{prooftree}
        \AxiomC{$\Gamma \Rightarrow_{\vec y} \Delta, \,\alpha$}
        \RightLabel{\ (L$\lnot$)}
        \UnaryInfC{$\lnot\alpha, \,\Gamma\Rightarrow_{\vec y}\Delta$}
    \end{prooftree}
\end{minipage}%
\begin{minipage}{0.5\textwidth}
    \begin{prooftree}
        \AxiomC{$\alpha,\,\Gamma\Rightarrow_{\vec y}\Delta$}
        \RightLabel{\ (R$\lnot$)}
        \UnaryInfC{$\Gamma \Rightarrow_{\vec y} \Delta,\,\lnot \alpha$}
    \end{prooftree}
\end{minipage}%
    
\noindent
\begin{minipage}{0.5\textwidth}
    \begin{prooftree}
        \AxiomC{$\Gamma\Rightarrow_{\vec y}\Delta, \,\alpha$}\AxiomC{$\beta,\,\Gamma\Rightarrow_{\vec y}\Delta$}
        \RightLabel{\ (L$\to$)}
        \BinaryInfC{$\alpha\to\beta,\,\Gamma\Rightarrow_{\vec y}\Delta$}
    \end{prooftree}
\end{minipage}%
\begin{minipage}{0.5\textwidth}
    \begin{prooftree}
        \AxiomC{$\alpha,\,\Gamma\Rightarrow_{\vec y}\Delta,\,\beta$}
        \RightLabel{\ (R$\to$)}
        \UnaryInfC{$\Gamma \Rightarrow_{\vec y} \Delta,\,\alpha\to\beta$}
    \end{prooftree}
\end{minipage}%

\smallskip

\begin{center}
    \fbox{Logical Rules for Quantifiers}
\end{center}

\noindent
\begin{minipage}{0.5\textwidth}
    \begin{prooftree}
        \AxiomC{$\zeta[t/x], \,\Gamma\Rightarrow_{\vec y}\Delta$}
         \RightLabel{\ (L$\forall$)}
        \UnaryInfC{$\forall x\,\zeta,\,\Gamma\Rightarrow_{\vec y}\Delta$}
    \end{prooftree}
\end{minipage}%
\begin{minipage}{0.5\textwidth}
    \begin{prooftree}
        \AxiomC{$\Gamma \Rightarrow_{(\vec y, x)} \Delta,\,\zeta$}
        \RightLabel{\ (R$\forall$)}
        \UnaryInfC{$\Gamma\Rightarrow_{\vec y}\Delta,\,\forall x\,\zeta$}
    \end{prooftree}
\end{minipage}%

\noindent
\begin{minipage}{0.5\textwidth}
    \begin{center}
        (with $t$ a term with variables in $\vec y$ substitutable for $x$; we stress that $x \notin \{y_1, \dots, y_n\}$.)
    \end{center}
\end{minipage}%
\begin{minipage}{0.5\textwidth}
    \begin{center}
        (we stress that $\Gamma:\vec y$, $\Delta:\vec y$, $x\notin\{y_1,\dots,y_m\}$)
    \end{center}
\end{minipage}%

\smallskip

\noindent
\begin{minipage}{0.5\textwidth}
    \begin{prooftree}
        \AxiomC{$\zeta, \,\Gamma\Rightarrow_{(\vec y, x)} \Delta$}
        \RightLabel{\ (L$\exists$)}
        \UnaryInfC{$\exists x\,\zeta,\,\Gamma\Rightarrow_{\vec y}\Delta$}
    \end{prooftree}
\end{minipage}%
\begin{minipage}{0.5\textwidth}
    \begin{prooftree}
        \AxiomC{$\Gamma\Rightarrow_{\vec y}\Delta, \,\zeta[t/x]$}
        \RightLabel{\ (R$\exists$)}
        \UnaryInfC{$\Gamma\Rightarrow_{\vec y}\Delta, \exists x\,\zeta$}
    \end{prooftree}
\end{minipage}%

\noindent
\begin{minipage}{0.5\textwidth}
    \begin{center}
        (we stress that $\Gamma:\vec y$, $\Delta:\vec y$, $x\notin\{y_1,\dots,y_m\}$)
    \end{center}
\end{minipage}%
\begin{minipage}{0.5\textwidth}
    \begin{center}
        (with $t$ a term with variables in $\vec y$ substitutable for $x$; we stress that $x \notin \{y_1, \dots, y_n\}$.)
    \end{center}
\end{minipage}%

\begin{center}
    We also assume that we can rename bound variables as appropriate.
\end{center}

\smallskip

If the language has equality, one adds the following two schemes.

\smallskip

\begin{center}
    \fbox{Rules for equality}
\end{center}
\begin{minipage}{0.5\textwidth}
    \begin{prooftree}
        \AxiomC{}
        \UnaryInfC{$\Rightarrow_{\vec y} t=t$}
    \end{prooftree}
\end{minipage}%
\begin{minipage}{0.5\textwidth}
    \begin{prooftree}
        \AxiomC{}
        \UnaryInfC{$t=u,\,\zeta[t/x]\Rightarrow_{\vec y}\zeta[u/x]$}
    \end{prooftree}
\end{minipage}%

\noindent
\begin{minipage}{0.5\textwidth}
\begin{center}
    (with $t$ a term with variables in $\vec y$)
    \end{center}
\end{minipage}%
\begin{minipage}{0.5\textwidth}
\begin{center}
    (with $t$, $u$ terms with var.\ in $\vec y$ substitutable for $x$)
    \end{center}
\end{minipage}%

\begin{remark}[Interpretation of formulas for a language with equality]\label{r:int-fmlas-=}
    
    Let $\L=(\mathbb{F},\bigcup_{n\in\N}\mathbb{P}_n)$ be a first-order language with equality, with $\mathbb{P}_n$ the set of $n$-ary predicate symbols. We consider an enumeration $x_1, x_2, \dots$ of the variables. Let $\P\colon\ctx\op\to\BA$ be an elementary first-order Boolean doctrine and let $(\mathbb{I}_n\colon\mathbb{P}_n\to\P(x_1,\dots,x_n))_{n\in\N}$ be a family of functions.
    We define inductively on the complexity of the formulas a function $\mathcal{I}$ from the set $\mathrm{Form}(\L)$
    of $\L$-formulas, resuming the definition from \cref{r:int-fmlas}:
    \begin{enumerate}[start=4]
        \item for all morphisms $ t\colon\vec x\to(y)$ and $ u\colon\vec x\to(y)$ in $\ctx$, we define
        \[\mathcal{I}( (t(\vec x)=u(\vec x)):\vec x)=\P(\ple{t,u})(\delta_{(y)}).\]
    \end{enumerate}
\end{remark}

\begin{definition}[Validity of a sequent in a first-order Boolean doctrine]\label{d:valid-sequent}
    Let $\L=(\mathbb{F},\bigcup_{n\in\N}\mathbb{P}_n)$ be a first-order language (with equality), with $\mathbb{P}_n$ the set of $n$-ary predicate symbols. Let $\Gamma\Rightarrow_{\vec y}\Delta$ be an $\L$-sequent. Let $\P\colon\ctx\op\to\BA$ be a(n elementary) first-order Boolean doctrine, let $(\mathbb{I}_n\colon\mathbb{P}_n\to\P(x_1,\dots,x_n))_{n\in\N}$ be a family of functions and let $\mathcal{I}$ be its extension in the sense of \cref{r:int-fmlas} (\cref{r:int-fmlas-=}).
    We say that $\Gamma\Rightarrow_{\vec y}\Delta$ is \emph{valid} in $(\P,\mathcal{I})$ if, in $\P(\vec y)$,
    \[
    \mathcal{I} \mleft(\bigwedge\Gamma:\vec y\mright)\leq\mathcal{I}\mleft(\bigvee\Delta:\vec y\mright).
    \]
\end{definition}

\begin{lemma}[The rules of calculus hold in first-order Boolean doctrines over categories of contexts]\label{l:rules-calc}
    Let $(\mathbb{F},\bigcup_{n\in\N}\mathbb{P}_n)$ be a first-order language (with equality), with $\mathbb{P}_n$ the set of $n$-ary predicate symbols. Let 
    \begin{prooftree}
        \AxiomC{$A_1 \quad \cdots \quad A_k$}
        \UnaryInfC{$C$}
    \end{prooftree}
    be an instance of a rule, with $A_1,\dots,A_k, C$ sequents. Let $\P\colon\ctx\op\to\BA$ be a(n elementary) first-order Boolean doctrine and $(\mathbb{I}_n\colon\mathbb{P}_n\to\P(\vec x))_{n\in\N}$ a family of functions such that $A_i$ is valid in $(\P,\mathcal{I})$ for all $i$, with $\mathcal{I}$ the extension of $(\mathbb{I}_n)_{n \in \N}$ as in \cref{r:int-fmlas} (\cref{r:int-fmlas-=}).
    Then, $C$ is valid in $(\P,\mathcal{I})$.
\end{lemma}

\begin{proof}
    The structural rules, the logical rules for connectives and the identity axioms are trivially satisfied.
    The cut rule is satisfied because of the lattice distributivity of each fiber of $\P$, and the context enlargement rule because of the monotonicity of $\P(\pr_1)\colon\P(\vec y)\to\P(\vec y,x)$.
    
    Let us prove that the logical rules for quantifiers hold.
    
    \noindent
    \begin{minipage}{0.5\textwidth}
        \begin{prooftree}
            \AxiomC{$\zeta[t/x], \,\Gamma\Rightarrow_{\vec y}\Delta$}
             \RightLabel{\ (L$\forall$)}
            \UnaryInfC{$\forall x\,\zeta,\,\Gamma\Rightarrow_{\vec y}\Delta$}
        \end{prooftree}
    \end{minipage}%
    \begin{minipage}{0.5\textwidth}
        \begin{prooftree}
            \AxiomC{$\Gamma \Rightarrow_{(\vec y, x)} \Delta,\,\zeta$}
            \RightLabel{\ (R$\forall$)}
            \UnaryInfC{$\Gamma\Rightarrow_{\vec y}\Delta,\,\forall x\,\zeta$}
        \end{prooftree}
    \end{minipage}%
    
    For (L$\forall$), let  $\gamma,\delta\in\P(\vec y)$, $\varphi\in\P(\vec y,x)$ and $t\colon\vec y\to (x)$ with  $\P(\ple{\id,t})(\varphi)\land\gamma\leq\delta$ in $\P(\vec y)$. Since $\P(\pr_1)\fa{(x)}{\vec y}\varphi\leq\varphi$ (counit of the adjunction) in $\P(\vec y,x)$, by applying $\P(\ple{\id,t})$ to both sides of the inequality we get $\fa{(x)}{\vec y}\varphi\leq\P(\ple{\id,t})(\varphi)$ in $\P(\vec y)$. Thus in $\P(\vec y)$ we have
    \[\fa{(x)}{\vec y}\varphi\land\gamma\leq\P(\ple{\id,t})(\varphi)\land\gamma\leq\delta.\]
    Taking $\varphi\coloneqq\mathcal{I}(\zeta:(\vec y,x))$, $\gamma\coloneqq\mathcal{I}(\bigwedge\Gamma:\vec y)$ and $\delta\coloneqq\mathcal{I}(\bigvee\Delta:\vec y)$, and using eq.~\eqref{eq:I-nat}, we get the desired result.
    
    For (R$\forall$), let $\gamma,\delta\in\P(\vec y)$ and $\varphi\in\P(\vec y,x)$ with $\P(\pr_1)(\gamma)\leq\P(\pr_1)(\delta)\lor\varphi$ in $\P(\vec y,x)$; then, in $\P(\vec y)$,
    \begin{equation}\label{eq:dual-frob}
    \gamma\leq\fa{(x)}{\vec y}(\P(\pr_1)(\delta)\lor\varphi)\leq\delta\lor\fa{(x)}{\vec y}\varphi.
    \end{equation}
    The first inequality follows by adjunction, while the second one is proved as follows:
    let $L\dashv R$ be an adjunction between Boolean algebras $A$ and $B$, with $L\colon B\to A$ a Boolean homomorphism and $R\colon A\to B$ a monotone function. Then, for all $a\in A$ and $b\in B$ we have $R(a\lor L(b))\leq R(a)\lor b$. Indeed,
    \[R(a\lor L(b))\leq R(a)\lor b\iff R(a\lor L(b))\land \lnot b\leq R(a)\iff LR(a\lor L(b))\land L(\lnot b)\leq a,
    \]
    which holds since   
    \[LR(a\lor L(b))\land L(\lnot b)\leq (a\lor L(b))\land L(\lnot b)=(a\land L(\lnot b))\lor(L(b)\land L(\lnot b))\leq a.
    \]
    This proves \eqref{eq:dual-frob}. Taking $\varphi\coloneqq\mathcal{I}(\zeta:(\vec y,x))$, $\gamma\coloneqq\mathcal{I}(\bigwedge\Gamma:\vec y)$ and $\delta\coloneqq\mathcal{I}(\bigvee\Delta:\vec y)$ gives the desired result.  
    The rules for the existential quantifiers are also satisfied with an analogous proof.

    If the language $\L$ has equality, we are left to show that the two rules for equality are satisfied.

    \noindent
    \begin{minipage}{0.5\textwidth}
        \begin{prooftree}
            \AxiomC{}
            \UnaryInfC{$\Rightarrow_{\vec y} t=t$}
        \end{prooftree}
    \end{minipage}%
    \begin{minipage}{0.5\textwidth}
        \begin{prooftree}
            \AxiomC{}
            \UnaryInfC{$t=u, \,\zeta[t/x] \Rightarrow_{\vec y}\zeta[u/x]$}
        \end{prooftree}
    \end{minipage}%
    
    The sequent $\Rightarrow_{\vec y} t=t$ is valid in $(\P,\mathcal{I})$ since
    \[\mathcal{I}(t=t:\vec y)=\P(\ple{t,t})(\delta_{(x)})=\P(t)(\P(\Delta_{(x)})(\delta_{(x)}))=\P(t)(\top_{\P(x)})=\top_{\P(\vec y)},\]
    where the second-to-last equality follows from the fact that $\P$ is elementary
    (see \cref{def:equality}\eqref{i:def=1}).
    
    We now show that the sequent $t=u,\,\zeta[t/x]\Rightarrow_{\vec y}\zeta[u/x]$ is valid in $(\P,\mathcal{I})$. First of all, observe that
    \[
    \mathcal{I}(t=u\land\zeta[t/x]:\vec y)=\mathcal{I}(t=u:\vec y)\land \mathcal{I}(\zeta[t/x]:\vec y)=\P(\ple{t,u})(\delta_{(x)})\land \P(\ple{\id_{\vec y},t})(\mathcal{I}(\zeta:(\vec y,x))).\]
    
    We write $\varphi\coloneqq\mathcal{I}(\zeta:(\vec y,x))$; by \cref{def:equality}\eqref{i:def=2}, in $\P(\vec y,x,{\vec y}',x')$ we have 
    \[\P(\ple{\pr_1,\pr_2})(\varphi)\land\delta_{(\vec y,x)}\leq\P(\ple{\pr_3,\pr_4})(\varphi). \]
    Applying $\P(\ple{\id_{\vec y},t,\id_{\vec y},u})$ on both sides, and recalling \cref{def:equality}\eqref{i:def=1} and \eqref{i:def=3}, we get in $\P(\vec y)$
    \[\P(\ple{\id_{\vec y},t})(\varphi)\land\P(\ple{t,u})(\delta_{(x)})\leq\P(\ple{\id_{\vec y},u})(\varphi), \]
    i.e.,
    \[
    \mathcal{I}(t=u\land\zeta[t/x]:\vec y)\leq\mathcal{I}(\zeta[u/x]:\vec y),
    \]
    and hence the sequent $t=u,\,\zeta[t/x]\Rightarrow_{\vec y}\zeta[u/x]$ is valid in $(\P,\mathcal{I})$.
\end{proof}

\section{A doctrine without a least quantifier-free fragment} \label{s:no-smallest}

We exhibit a first-order Boolean doctrine without a least quantifier-free fragment. The existence of such a first-order Boolean doctrine should come as no surprise, since the intersection of generating sets typically fails to be generating, and so algebraic structures typically fail to have a least generating set (or alike).

Consider the language consisting of an $n$-ary relation symbol $R_n$ for each $n \in \N$.
For each $n \in \N$, let $\alpha_n$ be the sentence
\[
    \forall x_1\,\dots\,\forall x_n\,(R_n(x_1,\dots,x_n)\leftrightarrow\exists x_{n +1}\, R_{n+1}(x_1,\dots,x_n,x_{n+1})),
\]
and define the theory $\T \coloneqq \{\alpha_n \mid n \in \N\}$.
We will show that the syntactic doctrine $\LT^\T$ has no least quantifier-free fragment (\cref{p:no-least-qff}).

The proof idea is the following: let $\mathbf{T}$ be the least Boolean subfunctor of $\LT^\T$, i.e.\ $\mathbf{T}(\{x_1, \dots, x_n\}) = \{ \top, \bot\}$.
We will prove that there is a decreasing sequence of quantifier-free fragments whose intersection is $\mathbf{T}$ (\cref{l:intersection-is-top-bot}).
In virtue of the fact that $\mathbf{T}$ is too ``small'' to be a quantifier-free fragment (\cref{f:does-not-generate}), we will then conclude that there is no least quantifier-free fragment of $\LT^\T$ (\cref{p:no-least-qff}).

Before we start, we give an intuition on the models of $\T$.
We can think of a model of $\T$ as a pair $(A, \mathcal{W})$, with $A$ (the ``alphabet'') a set and $\mathcal{W}$ (the ``language'') a set of finite words with letters in $A$ such that:
\begin{enumerate}
    \item $\mathcal{W}$ is closed under prefixes;
    \item every word in $\mathcal{W}$ can be properly extended to another word still in $\mathcal{W}$. Said differently, for every $w \in \mathcal{W}$ there is $w' \in \mathcal{W}$ of which $w$ is a proper prefix.
\end{enumerate}
A model $A$ of $\T$ can be understood as one such pair by taking the underlying set $A$ as the alphabet and defining the language $\mathcal{W}$ as the set whose words of length $n$ (for any $n \in \N$) are the words $a_1 \cdots a_n$ such that $(a_1, \dots, a_n)$ belongs to the interpretation of $R_n$ in the model $A$.
Conversely, a pair $(A, \mathcal{W})$ can be understood as a model of $\T$ by taking $A$ as the underlying set and defining the interpretation of $R_n$ (for each $n \in \N$) to be the set of tuples $(a_1, \dots, a_n)$ such that the word $a_1 \cdots a_n$ belongs to $\mathcal{W}$.

Let us first prove that $\mathbf{T}$ is not a quantifier-free fragment.

\begin{fact} \label{f:does-not-generate}
    $\mathbf{T}$ is not a quantifier-free fragment of $\LT^\T$. 
\end{fact}

\begin{proof}
    We prove that $\mathbf{T}$ does not generate $\LT^\T$, i.e.\ that the closure $\mathbf{Q}$ of the Boolean doctrine $\mathbf{T}$ under quantifiers is not $\LT^\T$.
    It is not difficult to prove that $\mathbf{Q}$ is $\{\top, \bot, \exists x\, \top, \lnot \exists x\, \top\}$ in each fiber (with appropriate dummy variables, and with some of these elements possibly coinciding). 
    To prove $\mathbf{Q} \neq \LT^\T$, for each $\alpha$ in $\{\top, \bot, \exists x\, \top, \lnot \exists x\, \top\}$ in context $\{x_1, \dots, x_k\}$ (to be thought of as dummy variables) we exhibit a model of $\T$ and an interpretation $\rho$ of the variables $\{x_1, \dots, x_k\}$ that satisfies $\alpha$ but not $R_0$, or vice versa.
    
    If $k = 0$, consider (i) the empty model with each $R_n(x_1, \dots, x_n)$ false, and (ii) a one-element model with each $R_n(x_1, \dots, x_n)$ true.
    These two models can be thought of as (i) the pair $(A,\mathcal{W})$ with the alphabet $A=\varnothing$ and the language $\mathcal{W}=\varnothing$ and (ii) the pair $(A,\mathcal{W})$ with the alphabet $A=\{a\}$ and the language $\mathcal{W} = \{a \cdots a \text{ ($n$ times)} \mid n \in \N\}$.
    Then $R_0$ is neither $\top$ nor $\exists x\, \top$ by (i), and is neither $\bot$ nor $\lnot \exists x \, \top$ by (ii).
    
    If $k \geq 1$, consider (i) a model with exactly an element and each $R_n(x_1, \dots, x_n)$ true (and the only possible interpretation of the variables $x_1, \dots, x_k$), and (ii) a model with exactly an element and every $R_n(x_1, \dots, x_n)$ false.
    These two models can be thought of as (i) the pair $(A,\mathcal{W})$ with the alphabet $A=\{a\}$ and the language $\mathcal{W} = \{a \cdots a \text{ ($n$ times)} \mid n \in \N\}$ and (ii) the pair $(A,\mathcal{W})$ with the alphabet $A=\{a\}$ and the language $\mathcal{W}=\varnothing$.
    Then $R_0$ is neither $\bot$ nor $\lnot \exists x\, \top$ by (i), and is neither $\top$ nor $\exists x\, \top$ by (ii).
    
    This proves that $R_0$ does not belong to $\mathbf{Q}$. Therefore, $\mathbf{Q} \neq \LT^\T$.
\end{proof}

We now define the sequence of quantifier-free fragments whose intersection will be claimed to be $\mathbf{T}$.
For each $n \in \N$, let $\P_0^n \colon \ctx\op \to \BA$ be the Boolean subfunctor of $\LT^\T$ ``generated'' by $\{R_m \mid m \geq n\}$, i.e.\ the closure of $\{R_m \mid m \geq n\}$ under Boolean operations and reindexings.
Concretely, for all $k \in \N$, $\P_0^n(x_1, \dots, x_k)$ consists of all $\T$-equivalence classes of Boolean combinations of formulas of the form $R_{m}(x_{f(1)}, \dots, x_{f(m)})$ for some $m \in \N$ with $m \geq n$ and some function $f \colon \{1, \dots, m\} \to \{1, \dots, k\}$.

\begin{fact}
    For each $n \in \N$, $\P_0^n$ is a quantifier-free fragment of $\LT^\T$.
\end{fact}

\begin{proof}
    $\P_0^n$ is clearly closed under Boolean combinations and reindexings.
    Moreover, each $R_m$ is in the first-order Boolean doctrine generated by $\P_0^n$, because if $m \geq n$ then it belongs to $\P_0^n$, and if $m < n$ then it is first-order definable from $R_{n}$.
\end{proof}

Now, the main effort is proving that the intersection of all $\P_0^n$ (for $n \in \N$) is $\mathbf{T}$.
Here is the key lemma.

\begin{lemma} \label{l:prefix-lemma}
    Let $k \in \N$ (and $x_1, \dots, x_k$ be different variables), let $\bar{i}, \bar{j}, m_1,\dots,m_{\bar{i}},n_1,\dots,n_{\bar{j}}\in \N$, for every $i \in \{1, \dots, \bar{i}\}$ let $f_i \colon \{1, \dots, m_{{i}}\} \to \{1, \dots, k\}$ be a function, and for every $j \in \{1, \dots, \bar{j}\}$ let $g_j \colon \{1, \dots, n_{{j}}\} \to \{1, \dots, k\}$ be a function.
    The inequality
    \[
    \bigwedge_{i = 1}^{\bar{i}} R_{m_i}(x_{f_{i}(1)}, \dots, x_{f_{i}( m_i)})
    \leq 
    \bigvee_{j = 1}^{\bar{j}} R_{n_j}(x_{g_{j}(1)}, \dots, x_{g_{j}(n_j)}).
    \]
    in $\LT^\T(\{x_1, \dots, x_k\})$ holds if and only if there is $i_0 \in \{1, \dots, \bar{i}\}$ and $j_0 \in \{1, \dots, \bar{j}\}$ such that the tuple $(x_{g_{j_0}(1)}, \dots, x_{g_{j_0}( n_{j_0})})$ is a prefix of the tuple $(x_{f_{i_0}(1)}, \dots, x_{f_{i_0}(m_{i_0})})$.
\end{lemma}

\begin{proof}
    The right-to-left implication is straightforward.

    We prove the contrapositive of the left-to-right implication.
    Suppose that for no $i$ and $j$ the tuple $(x_{g_{j}(1)}, \dots, x_{g_{j}( n_{j})})$ is a prefix of $(x_{f_{i}(1)}, \dots, x_{f_{i}(m_{i})})$.
    We shall prove
    \[
    \bigwedge_{i = 1}^{\bar{i}} R_{m_i}(x_{f_{i}(1)}, \dots, x_{f_{i}( m_i)})
    \nleq 
    \bigvee_{j = 1}^{\bar{j}} R_{n_j}(x_{g_{j}(1)}, \dots, x_{g_{j}(n_j)}).
    \]
    To do so, we exhibit a model that satisfies the formula on the left but not the one on the right.
    Consider the alphabet $A = \{x_1, \dots, x_k, c\}$, where $c$ is an element different from $x_1, \dots, x_k$.
    Let $\mathcal{W}$ be the language whose words are any prefix of $x_{f_{i}(1)} \cdots x_{f_{i}(m_{i})}$ and any word of the form 
    \[
    x_{f_{i}(1)} \cdots x_{f_{i}(m_{i})} \underbrace{c\cdots c}_{l \text{ times}}
    \]
    for any $i$ and $l \in \N$.
    We now consider the pair $(A, \mathcal{W})$ as a model of $\T$, together with the variable assignment $\rho$ mapping each free variable $x_i$ to the element $x_i \in A$.
    Since $x_{f_{i}(1)} \cdots x_{f_{i}(m_{i})}$ belongs to $\mathcal{W}$ for each $i$, 
    \[
    (A, \mathcal{W}, \rho) \vDash \bigwedge_{i = 1}^{\bar{i}} R_{m_i}(x_{f_{i}(1)}, \dots, x_{f_{i}( m_i)}).
    \]
    Since for no $i$ and $j$ the tuple $(x_{g_{j}(1)}, \dots, x_{g_{j}( n_{j})})$ is a prefix of $(x_{f_{i}(1)}, \dots, x_{f_{i}(m_{i})})$, and since $c$ differs from all $x_i$'s, none of the words $x_{g_{j}(1)} \cdots x_{g_{j}( n_{j})}$ belongs to $\mathcal{W}$, and so $(A, \mathcal{W}, \rho) \not\vDash \bigvee_{j = 1}^{\bar{j}} R_{n_j}(x_{g_{j}(1)}, \dots, x_{g_{j}(n_j)})$.
\end{proof}

\begin{lemma} \label{l:intersection-is-top-bot}
    For every $k \in \N$ we have $\bigcap_{n \in \N} \P_0^n(\{x_1, \dots, x_k\}) = \{ \top, \bot\}$.
\end{lemma}

\begin{proof}
    The right-to-left inclusion is immediate.

    Let us prove the left-to-right inclusion.
    Let $\varphi(x_1, \dots, x_k) \in \bigcap_{n \in \N} \P_0^n(\{x_1, \dots, x_k\})$.
    We shall prove $\varphi(x_1, \dots, x_k) \in \{\top, \bot\}$.
    Suppose, by way of contradiction, $\varphi(x_1, \dots, x_k) \notin \{\top, \bot\}$, and let us look for a contradiction.
    
    Since $\varphi(x_1, \dots, x_k)$ belongs to $\P_0^0$, $\varphi(x_1, \dots, x_k)$ is a Boolean combination of elements of the form $R_{m}(x_{f(1)}, \dots, x_{f(m)})$ for some $m \in \N$ and some function $f \colon \{1, \dots, m\} \to \{1, \dots, k\}$. Therefore, it can be written as a finite disjunction of finite conjunctions of elements of the form above and their negations.
    Since $\varphi(x_1, \dots, x_k) \neq \bot$, one of the disjuncts has to be different from $\bot$.
    Therefore, we have
    \begin{equation} \label{eq:left}
        \bot \neq \mleft(\bigwedge_{i = 1}^{\bar{i}} R_{m_i}(x_{f_i(1)}, \dots, x_{f_i(m_i)})\mright) \land \mleft(\bigwedge_{j = 1}^{\bar{j}} \lnot R_{n_j}(x_{g_j(1)}, \dots, x_{g_j(n_j)})\mright) \leq \varphi(x_1, \dots, x_k),
    \end{equation}
    for some $\bar{i}, \bar{j}, (m_i)_i, (n_j)_j, (f_i \colon \{1, \dots, m_i\} \to \{1, \dots, k\})_i, (g_j \colon \{1, \dots, n_j\} \to \{1, \dots, k\})_j$.

    We next claim that we can remove the positive occurrences from \eqref{eq:left}; this is the point where we make full use of the hypothesis that $\varphi$ belongs to \emph{every} $\P_0^n$.
    
    \begin{claim} \label{cl:positive-removed}
        $\bigwedge_{j = 1}^{\bar{j}} \lnot R_{n_j}(x_{g_j(1)}, \dots, x_{g_j(n_j)}) \leq \varphi(x_1, \dots, x_k)$.
    \end{claim}

    \begin{proof}[Proof of claim]
        Let $\bar{m}$ be any natural number strictly greater than $m_1, \dots, m_{\bar{i}}$, $n_1, \dots, n_{\bar{j}}$.
        Then 
        \[
        \varphi(x_1, \dots, x_k) \in \bigcap_{n \in \N} \P_0^n(\{x_1, \dots, x_k\}) \subseteq \P_0^{\bar{m}}(\{x_1, \dots, x_k\}).
        \]
        Therefore, $\varphi(x_1, \dots, x_k)$ is a Boolean combination of elements of the form $R_{m}(x_{f(1)}, \dots, x_{f(m)})$ for some $m \geq \bar{m}$ and some function $f \colon \{1, \dots, m\} \to \{1, \dots, k\}$.
        Thus, $\varphi(x_1, \dots, x_k)$ is a finite conjunction of finite disjunctions of elements of the form above and their negations.
        Moreover, we can assume each conjunct not to be $\top$.
        It is enough to prove \cref{cl:positive-removed} for each conjunct of $\varphi$ in place of $\varphi$.
        Let
        \[
            \mleft(\bigvee_{l = 1}^{\bar{l}} R_{p_l}(x_{r_l(1)}, \dots, x_{r_l(p_l)})\mright) \lor \mleft(\bigvee_{h = 1}^{\bar{h}} \lnot R_{q_h}(x_{s_h(1)}, \dots, x_{s_h(q_h)})\mright)
        \]
        be one such conjunct.

        We then have
        \begin{align*}
            &\mleft(\bigwedge_{i = 1}^{\bar{i}} R_{m_i}(x_{f_i(1)}, \dots, x_{f_i(m_i)})\mright) \land \mleft(\bigwedge_{j = 1}^{\bar{j}} \lnot R_{n_j}(x_{g_j(1)}, \dots, x_{g_j(n_j)})\mright)\\
            & \leq \varphi(x_1, \dots, x_k)\\
            & \leq \mleft(\bigvee_{l = 1}^{\bar{l}} R_{p_l}(x_{r_l(1)}, \dots, x_{r_l(p_l)})\mright) \lor \mleft(\bigvee_{h = 1}^{\bar{h}} \lnot R_{q_h}(x_{s_h(1)}, \dots, x_{s_h(q_h)})\mright).
        \end{align*}
        Therefore,
        \begin{align*}
            &\mleft(\bigwedge_{i = 1}^{\bar{i}} R_{m_i}(x_{f_i(1)}, \dots, x_{f_i(m_i)})\mright) 
            \land 
            \mleft(\bigwedge_{h = 1}^{\bar{h}} R_{q_h}(x_{s_h(1)}, \dots, x_{s_h(q_h)})\mright)\\
            &\leq 
            \mleft(\bigvee_{l = 1}^{\bar{l}} R_{p_l}(x_{r_l(1)}, \dots, x_{r_l(p_l)})\mright) 
            \lor 
            \mleft(\bigvee_{j = 1}^{\bar{j}} R_{n_j}(x_{g_j(1)}, \dots, x_{g_j(n_j)})\mright).
        \end{align*}
        Therefore, by \cref{l:prefix-lemma}, one of the following cases applies.
        \begin{enumerate}
            \item \label{i:impossible-by-length}
            There are $l_0 \in \{1, \dots, \bar{l}\}$ and $i_0 \in \{1, \dots, \bar{i}\}$ such that the tuple $(x_{r_{l_0}(1)}, \dots, x_{r_{l_0}(p_{l_0})})$ is a prefix of $(x_{f_{i_0}(1)}, \dots, x_{f_{i_0}(m_{i_0})})$.

            \item \label{i:impossible-by-non-top}
            There are $l_0 \in \{1, \dots, \bar{l}\}$ and $h_0 \in \{1, \dots, \bar{h}\}$ such that the tuple $(x_{r_{l_0}(1)}, \dots, x_{r_{l_0}(p_{l_0})})$ is a prefix of $(x_{s_{h_0}(1)}, \dots, x_{s_{h_0}(q_{h_0})})$.

            \item \label{i:impossible-by-non-bot}
            There are $j_0 \in \{1, \dots, \bar{j}\}$ and $i_0 \in \{1, \dots, \bar{i}\}$ such that the tuple $(x_{g_{j_0}(1)}, \dots, x_{g_{j_0}(n_{j_0})})$ is a prefix of $(x_{f_{i_0}(1)}, \dots, x_{f_{i_0}(m_{i_0})})$.

            \item \label{i:possible}
            There are $j_0 \in \{1, \dots, \bar{j}\}$ and $h_0 \in \{1, \dots, \bar{h}\}$ such that the tuple $(x_{g_{j_0}(1)}, \dots, x_{g_{j_0}(n_{j_0})})$ is a prefix of $(x_{s_{h_0}(1)}, \dots, x_{s_{h_0}(q_{h_0})})$.
        \end{enumerate}
        
        The case \eqref{i:impossible-by-length} is not possible since $p_{l_0} \geq \bar{m} > m_{i_0}$. (This is where we use our hypothesis on $\bar{m}$.)

        \eqref{i:impossible-by-non-top} is not possible since the conjunct $\mleft(\bigvee_{l = 1}^{\bar{l}} R_{p_l}(x_{r_l(1)}, \dots, x_{r_l(p_l)})\mright) \lor \mleft(\bigvee_{h = 1}^{\bar{h}} \lnot R_{q_h}(x_{s_h(1)}, \dots, x_{s_h(q_h)})\mright)$ of $\varphi$ is not $\top$, by hypothesis.

        \eqref{i:impossible-by-non-bot} is not possible as $\mleft(\bigwedge_{i = 1}^{\bar{i}} R_{m_i}(x_{f_i(1)}, \dots, x_{f_i(m_i)})\mright) \land \mleft(\bigwedge_{j = 1}^{\bar{j}} 
        \lnot R_{n_j}(x_{g_j(1)}, \dots, x_{g_j(n_j)})\mright)\neq \bot$ (\eqref{eq:left}).

        Therefore, \eqref{i:possible} holds.
        This proves our claim.
    \end{proof}

    Analogously, there are $\bar{l}$, $(p_l)_l$ and $(r_l \colon \{1, \dots, p_l\} \to \{1, \dots, k\})_l$ such that
    \begin{equation} \label{eq:dual-of-positive-removed}
        \varphi(x_1, \dots, x_k) \leq \bigvee_{l = 1}^{\bar{l}} R_{p_l}(x_{r_l(1)}, \dots, x_{r_l(p_l)}).
    \end{equation}
    By \cref{cl:positive-removed} and \eqref{eq:dual-of-positive-removed}, we have
    $
        \bigwedge_{j = 1}^{\bar{j}} \lnot R_{n_j}(x_{g_j(1)}, \dots, x_{g_j(n_j)}) \leq \bigvee_{l = 1}^{\bar{l}} R_{p_l}(x_{r_l(1)}, \dots, x_{r_l(p_l)})
    $,
    and thus $\top \leq \mleft(\bigvee_{l = 1}^{\bar{l}} R_{p_l}(x_{r_l(1)}, \dots, x_{r_l(p_l)})\mright) \lor \mleft(\bigvee_{j = 1}^{\bar{j}} R_{n_j}(x_{g_j(1)}, \dots, x_{g_j(n_j)}) \mright)$,
    which contradicts \cref{l:prefix-lemma}.
\end{proof}

\begin{proposition}\label{p:no-least-qff}
    The first-order Boolean doctrine $\LT^\T$ has no least quantifier-free fragment.
\end{proposition}

\begin{proof}
    Suppose a least quantifier-free fragment $\mathbf{Q}$ exists.
    By \cref{l:intersection-is-top-bot}, $\mathbf{T}$ is an intersection of quantifier-free fragments, and so $\mathbf{Q} \subseteq \mathbf{T}$, which implies $\mathbf{Q} = \mathbf{T}$, contradicting \cref{f:does-not-generate}.
\end{proof}

\section*{Acknowledgments}
We would like to thank the logic group at the University of Salerno for supporting a sequence of seminars in logic that included a workshop on doctrines, which allowed us to advance this collaboration.
Significant progress in our work occurred during a visit at the University of Luxembourg, for which we are grateful to Bruno Teheux.
Moreover, we would like to express our heartfelt thanks to Maria Emilia Maietti and Luca Reggio for their invaluable feedback and suggestions.
We also thank the participants of a seminar at the University of Padua for their engaging questions and comments, and Jérémie Marquès for drawing our attention to the fact that first-order Boolean doctrines form a variety, rather than merely a quasi-variety—a point that was independently emphasized by the referee. This observation has been incorporated into \cref{r:many-sorted}.
We thank Silvio Ghilardi, Vincenzo Marra and Davide Trotta for suggestions on the literature. 
The first author also thanks Achim Jung for various discussions on the topic. 
Finally, we express our deepest gratitude to the referee for their careful and thorough reading of our manuscript and their many insightful suggestions, which have significantly improved the quality of this paper.
For example, the elegant and constructive proof of \cref{t:bd-embeds-fo}, which replaced our earlier approach based on models, is due to the referee.
We feel fortunate to have received a report of such high quality.

\textit{Funding.}
The first author's research was supported by the UK Research and Innovation (UKRI) under the UK government’s Horizon Europe funding guarantee (Project ``DCPOS'', grant number EP/Y015029/1) and by the Italian Ministry of University and Research through the PRIN project n.~20173WKCM5 \emph{Theory and applications of resource sensitive logics}. The second author's research was supported by the Luxembourg National Research Fund (FNR), ``COMOC'' Project (ref. C21/IS/16101289).

\bibliographystyle{plain}
\bibliography{Biblio}
\end{document}